\documentclass[review]{elsarticle}

\usepackage{lineno,hyperref}
\usepackage{amsmath, amsfonts}
\usepackage{graphicx,graphbox}
\usepackage[counterclockwise, figuresleft]{rotating}
\usepackage{longtable} %longtable package
\usepackage{amsmath}
\usepackage{amssymb}
\usepackage{amsthm}
\usepackage{natbib}
\usepackage{float}
\usepackage{graphics}
\usepackage{xcolor}
\usepackage{enumerate} % must be put before enumitem
\usepackage[shortlabels]{enumitem}
\usepackage{xcolor}
\usepackage{makecell}
\usepackage{subfig} % for putting figures side by side

\hypersetup{pdfauthor={Name}}
\hypersetup{
	colorlinks,
	linkcolor={red!50!black},
	citecolor={blue!50!black},
	urlcolor={blue!80!black}
}

\newcommand{\ben}{\begin{enumerate}}
\newcommand{\een}{\end{enumerate}}
\def\bco{\iffalse}
\newtheorem{thm}{Theorem}
\newcommand{\bthm}{\begin{thm}}
\newcommand{\ethm}{\end{thm}}
\newtheorem{prop}{Proposition}

\newtheorem{cor}{Corollary}
{
	\theoremstyle{remark}
	\newtheorem{rem}{Remark}

}

\usepackage{titlesec}

\newcommand\unappendix{
	\titleformat{\section}{\centering\bf\large}{}{0em}{}
}

\def \Lp {\mathcal{L}_p}

\usepackage{amsmath,bm,booktabs,hyperref,mathtools,natbib,xspace}
\usepackage{url,xcolor}

\newcommand{\bal}{\begin{equation}\aligned}
\newcommand{\eal}{\endaligned\end{equation}}
\newcommand{\bgt}{\begin{equation}\begin{gathered}}
\newcommand{\egt}{\end{gathered}\end{equation}}

\DeclareMathOperator*{\argmin}{argmin}

\def\case{X}
\def\diffop{\mathrm{d}}
\def\elag{\widetilde} % estimate for lag selection
\def\elagVslope{\elag\beta^{*}}
\def\evs{\check} % estimate for variable selection
\def\ehislope{\widehat{\hislope}}
\def\err{\epsilon}

\def\hislope{\gamma}
\def\intc{\alpha}
\def\lag{\tau}
\def\nn{\nonumber}
\def\nprdt{J}
\def\nsub{n}
\def\ntgrid{K}
\def\prdt{U}
\def\pidx{j}
\def\pnty{\lambda}
\def\psel{p}

\def\slope{\beta}
\def\sidx{i}
\def\tdom{\mathcal{T}}
\def\tidx{k}
\def\tm{t}
\def\vtm{s}
\def\vslope{\beta^{*}}
\modulolinenumbers[5]

\journal{Journal of Mathematical Analysis and Applications}

\biboptions{authoryear}
\begin{document}

\begin{frontmatter}

\title{Learning Delay Dynamics for Multivariate  Stochastic Processes, With Application to the Prediction of the Growth Rate of  COVID-19 Cases in the United States\footnote{All authors declare no financial or other competing interests. The data for this study was accessed from \url{https://github.com/OpportunityInsights/EconomicTracker} \citep{economicTracker} on September $19$, $2020$. The authors are not responsible for collecting the data, declare no financial interests  and maintain ethical neutrality.}}

\author[1]{\normalsize Paromita Dubey}
\author[2]{Yaqing Chen\footnote{Contributed equally to the paper.}}
\newcommand\CoAuthorMark{\footnotemark[\arabic{footnote}]}
\author[2]{\'Alvaro Gajardo\protect\CoAuthorMark}
\author[2]{Satarupa~Bhattacharjee\protect\CoAuthorMark}
\author[2]{Cody Carroll\protect\CoAuthorMark}
\author[2]{Yidong Zhou\protect\CoAuthorMark}
\author[2]{Han Chen\protect\CoAuthorMark}
\author[2]{Hans-Georg M\"uller\footnote{Corresponding author. Email address: hgmueller@ucdavis.edu. Research supported in part by NSF Grant DMS-201426.}}
\address[1]{\normalsize Department of Data Sciences and Operations,\\ USC Marshall School of Business}
\address[2]{\normalsize Department of Statistics, University of California, Davis}
%\address{University of California, Davis}
%\fntext[myfootnote]{Contributed equally to the paper}




\begin{abstract}
Delay differential equations form the underpinning of many complex dynamical systems. The forward problem of solving random differential equations with delay has received increasing attention in recent years. Motivated by the challenge to predict the COVID-19 caseload trajectories for individual states in the U.S., we target here the inverse problem. Given a sample of observed random trajectories obeying an unknown random differential equation model with delay, we use a functional data analysis framework to learn the model parameters that govern the underlying dynamics from the data. We show existence and uniqueness of the analytical solutions of the population delay random differential equation model when one has discrete time delays in the functional concurrent regression model and also for a second scenario where one has a delay continuum or distributed delay. The latter involves a functional linear regression model with history index. The derivative of the process of interest is modeled using the process itself as predictor and also other functional predictors with predictor-specific delayed impacts. This dynamics learning approach is shown to be well suited to model the growth rate of COVID-19 for the states that are part of the U.S., by pooling information from the individual states, using the case process and concurrently observed economic and mobility data as predictors.

\end{abstract}


\end{frontmatter}

%\linenumbers

\section{Introduction}

The modeling of time dynamical systems is of interest in  multiple scientific fields. While ordinary differential equations (ODE) have a  long history, interest in delay differential equations (DDE) is more recent, albeit both  have been extensively studied \citep{codd:55, weis:67,ablo:75,driv:12,asl:03}. A  DDE is a natural extension of an ODE when observed processes have an aftereffect. Unlike the situation for ODEs, the solution of a DDE depends not only on the initial state but on the entire history of the process in a time interval of length equal to the delay prior to the initial time point. 

Moving beyond the basic notion of a deterministic ODE,  random differential equations (RDE) \citep{stra:70,soon:73, cort:07, neck:13} and stochastic differential equations (SDE) \citep{arno:74} are used to accommodate probabilistic uncertainty in temporal stochastic processes. In this paper we focus on RDEs, where the random effects are manifested in the model parameters. These include coefficients, initial conditions and forcing terms that are typically smooth in time, leading  to differentiable sample path solutions of the RDE, in contrast to the situation for SDEs,  which have a non-differentiable forcing term, typically a Wiener process, and therefore require stochastic calculus \citep{ito:51}.

Extensive developments in the theory of the forward problem of obtaining solutions for RDEs stand in contrast with statistical approaches, which include data-oriented methodology for the inverse problem, that is, learning the nature of the  differential equation from data.  This approach has been referred to as empirical dynamics \citep{mull:10}. The starting point is a sample of random trajectories that are viewed as independent and identically distributed (i.i.d.) realizations of an underlying smooth (continuously differentiable) stochastic process. For a smooth stochastic process $X(t)$, there always exists a function $f$ such that
\begin{align}
    E(X'(t)|X(t))&=f(t,X(t)) \\ X'(t)&=f(t,X(t))+Z(t) \label{ede}
\end{align}
with $E(Z(t)|X(t))=0$ almost surely. This forms the basis of empirical dynamics and  has led to both parametric and nonparametric modeling approaches for $f$ \citep{zhu:11,verz:12}, which can be characterized as dynamics learning from data. Due to the minimal assumptions, dynamics learning covers a large number of specific ODEs that do not need to be a priori specified, however it  requires that a sample repeated realizations of the underlying stochastic process is available. Such samples form the backbone of  functional data analysis \citep{rams:05:1,jlwang:16}. A pertinent  example is provided by COVID-19 caseload trajectories, where one observes samples of such trajectories across geographic units such as states or countries \citep{carr:20}. 

Related but quite distinct statistical inference concerns the 
problem of identifying the parameters of an a priori  specified dynamic system \citep{brun:08,li:05,chen:08}, usually just from one or very few realizations. Here the trajectories are considered non-stochastic but are typically  assumed to have been measured with noise. This  leads to a curve fitting problem that in simple cases can be addressed with nonlinear least squares, but for which also  nonparametric and semiparametric statistical methods have been employed \citep{lian:08, paul:11}. A key difference is that in dynamics learning one requires data that can be considered as an independent and identically distributed sample of realizations of an underlying smooth stochastic process $X_i(t), i=1,\dots,n$, while the parameter identification problem is often addressed
given (noisy) data from one trajectory.  The  modeling approach that we introduce here follows the paradigm of dynamics learning, albeit in a somewhat more structured framework than empirical dynamics. 

%Over the years, DDEs have been thoroughly explored.  
The general form of DDEs for vector functions is 
\begin{equation*}
    \frac{d}{dt}\mathbf{x}(t) = f(t,\mathbf{x}(t),\mathbf{x}_\tau(t)),
\end{equation*}
where $\mathbf{x}(t)$ is the function value or state  at time $t$ and $\mathbf{x}_\tau(t)$ can be either a vector of the states evaluated at discrete time delays $\tau_i \ge 0,  \, i=1,...,m$, so that $\mathbf{x}_\tau(t) = (x(t-\tau_0),x(t-\tau_1),\dots,x(t-\tau_m))$, or
alternatively an integral of the trajectory $\mathbf{x}(\cdot)$ over a past period, representing a continuum of delays often called distributed delays \citep{bell:13, jaco:93, bjor:03, elna:89, mehr:14, cara:18,bere:01,gurn:80,aziz:16}. Results on existence and uniqueness of random differential equations with delay (RDED) are relatively recent, where  \cite{cala:19} studied conditions for $\mathcal{L}_p$ existence and uniqueness of the solutions of a RDED and  \cite{cort:20} the specific case of a linear RDED with a forcing term.

%Random and  DDEs have numerous applications in different areas, such as electronic and transportation systems \citep{jams:84}, epidemiological modelling \citep{cook:96, bere:01, arin:06}, chemical processes \citep{zhu:19}, biological or immunological dynamics \citep{shai:16, cord:86}. \\

In this paper we propose dynamics learning from a sample of multivariate functional data, where the model generating the observed derivative trajectories is assumed to be a RDE with a delay component. The randomness in the DE is included in the forcing term, part of which is explained by additional covariates. These  are stochastic processes with individual delay components, which in the motivating  COVID-19 application correspond to trajectories of mobility and  economic activity. The  unexplained remainder is a drift process, which %was also considered in empirical dynamics \citep{mull:10} and 
appears as  $Z(t)$  in equation \eqref{ede}.   

We utilize tools from the theory of RDE to provide a foundation for the proposed population models and   establish that all the models described in Section \ref{sec: model} correspond to RDED with unique solutions, which depend on the model parameters. The goal is then to learn these unknown parameters, which govern the dynamics encapsulated in the population RDED, by pooling information across the sample of observed trajectories. The presence of  a large number of covariates and delay components to choose from gives rise to the challenge of model selection. To address this we propose an initial pruning step as described in Section \ref{sec:varSelect} followed by a backfitting step to optimize over a set of delay components as outlined in Section \ref{sec:backfitting}. The  proposed approach differs from existing optimization strategies  \citep{wang:12, zhou:16,mehr:14,wang:20} and statistical methodology  \citep{jarn:17}
for parameter estimation that has been previously deployed  to identify a DDE based on noisy data observed for  a single non-stochastic trajectory. To our knowledge, 
dynamics learning where one has samples of stochastic processes has so far not been explored for the case of an underlying  RDED, even for the case of a one-dimensional stochastic process. We furthermore demonstrate in this paper that such an approach is well suited for modeling the growth rates of COVID-19.

\bco
Starting with a vector of co-evolving stochastic processes, we borrow information across the sample of observed trajectories to learn the dynamics encapsulated in the population RDED. The  proposed approach differs fundamentally from the existing optimization   \citep{wang:12, zhou:16,mehr:14,wang:20} and statistical methodology  \citep{jarn:17}
for parameter estimation that has been previously deployed  to identify a DDE based on noisy data observed for  a single non-stochastic trajectory. To our knowledge, 
dynamics learning where one has samples of stochastic processes has so far not not been explored for the case of an underlying  RDED, even for the case for a one-dimensional stochastic process. We furthermore demonstrate below that such an approach is well suited for modeling the growth rates of COVID-19. 
\fi

We consider both discrete and distributed delay models and establish existence and uniqueness of the solutions in the $\Lp$ sense \citep{cala:19} in Section \ref{sec: solutions}.  For the case of  discrete delays we harness  functional concurrent regression models \citep{cai:00,huan:04,sent:08,sent:11,huan:04} and for distributed delays  the functional history-index model \citep{malf:03,sent:10}, which incorporates a range of recent past values of the process. % and the predictors. We study existence and uniqueness of the solutions of both the popuation RDED models above. As far as we are aware, this forward problem has not been previously explored for distributed delay models.  
For dynamics learning of RDEDs, we  adopt a two stage procedure, where we first utilize functional linear regression  \citep{card:99,yao:05,morr:15} with history index to learn the distributed delay, where the regression parameter function then  corresponds to a  history index function for the process of interest. In a second step the resulting linear predictor, which is the inner product of the history index function and the predictor process of interest, is  used as a  predictor along with additional covariate processes with covariate-specific delays in a concurrent model to  fit the derivative process. 

We apply this model to predict the time-dynamic growth rate of COVID-19 for individual states in the United States, where the sample of all states and their case trajectories provides the sample of trajectories that is the starting point for the proposed dynamics learning. Predictor processes that we consider include the cumulative case (caseload) process, daily economic indicators and changes in mobility patterns. Modeling  the time evolution dynamics of the COVID-19 pandemic is of great importance to understand and interpret the underlying associations as well as for deploying resources and formulating policies in the face of great  uncertainty \citep{bret:20,bert:20, hao:20}. The proposed methodology is also of interest to assess the dynamics of many other empirically observed complex multivariate stochastic processes for which samples of observed trajectories are available.  We show by means of leave-one-out predictions that by employing dynamics learning for  the proposed RDED model with  delay components one obtains  considerably more accurate time-dynamic growth rate predictions  for COVID-19 caseload curves compared to  models without delays, demonstrating the importance of considering the inclusion of lags when modeling the growth rate of COVID-19.

\section{Proposed Models}

\label{sec: model}

\noindent Let $\left(X(\cdot),\mathbf{U}(\cdot)\right)$ denote a multivariate stochastic process where $X(\cdot)$ is a continuously differentiable  process of interest, $\mathbf{U}(\cdot)=(U_1(\cdot), \dots, U_J(\cdot))^\intercal$  is a vector function of  additional covariates, %Let $X'(\cdot)$ be the derivative of $X(\cdot)$. 
and $[t_0,T]$ is a time window of interest.    Consider a  RDED with discrete delays, 
\begin{align}
\label{CONC:DEF}
\frac{dX(t)}{dt} &= \alpha(t)+ \beta_0(t)X(t-\tau_0)+ \sum_{j=1}^J\beta_{j}(t)U_j(t-\tau_{j})+Z(t) \nonumber \\
:&=f(X(t-\tau_0),t),  \quad t \in [t_0,T], \nonumber \\
X(t) &= g(t) ,\ t \in [t_0-\tau_0, t_0],
\end{align}
where $g$ corresponds to the initial condition stochastic process, the  $\tau_j, \, j=0,\dots,J,$ are discrete delays, and $\alpha(t),\beta_0(t),\beta_j(t)$ are smooth functions whose regularity will be specified below in Section~\ref{sec: solutions}. In the above, $Z(\cdot)$ is a random drift process that is  independent of $\left(X(\cdot),\mathbf{U}(\cdot)\right)$. 

While the inclusion of discrete delays is an extension of the classical functional concurrent regression model \citep{rams:05:1}, the functional linear regression model with history index \citep{malf:03,sent:10} might be a better choice when the derivative trajectories depend not only on the predictor process value at a single past instant but on the entire continuum in the recent past. We model these distributed delays as  
\begin{align}
\label{histIndex1}
\frac{dX(t)}{dt} &= \alpha(t) +\int_0^{\tau_0}\gamma(s,t) X(t-s)\ ds \nonumber \\ &  + \int_0^{\tau_1} \gamma_1(s,t) U(t-s)\  ds +Z(t), \quad t \in [t_0,T] \nonumber \\ X(t) &= g(t), \ t\in[t_0-\tau_0,t_0],
\end{align}
where $g$ again  is an initial condition process. For the purpose of illustration and technical derivations, we assume that  $U(\cdot)$ is a univariate process in \eqref{histIndex1}; the corresponding multivariate generalization is straightforward.

\bco 
\begin{align}
\label{histIndex2}
\frac{dX(t)}{dt} &= \alpha(t) +
  \int_0^{\tau_0}\gamma(s,t)X(t-s)ds  \nonumber \\ &  + \sum_{j=1}^J\int_0^{\tau_{j}}\gamma_{j}(s,t)U_j(t-s)ds+Z(t)  \quad t \in [t_0,T] \nonumber\\ 
  X(t) &=  g(t), \ t\in[t_0-\tau_0,t_0].
\end{align}
In \eqref{CONC:DEF}, $\beta_j(\cdot)$, $j=0, \dots, J$, are the varying coefficient functions and in \eqref{histIndex2}, $\gamma(\cdot), \gamma_j(\cdot)$, $j=1, \dots, J$, are the history index coefficient functions, and $\tau_j$, $j=0, \dots, J$, are the corresponding lags. %Identifiability is ensured by the assumption that $\tau_j=\tau+\alpha_j$ with the constraint $\sum_{j=1}^{p+1}\alpha_j=0$.
In the absence of covariates, the  model
\bal
\label{eq:history-nocov}
X'(t) = & \alpha(t) +\int_0^{\tau_0}\gamma(s,t) X(t-s)\ ds +Z(t), \ t \in [t_0,T],
\\ X(t)= & g(t), \quad t\in[t_0-\tau_0,t_0]
\eal
emerges as  a special case of equation \eqref{histIndex2}.

When the coefficient functions $\gamma_j(s,t)$ in \eqref{histIndex2} are separable in $s$ and $t$, one obtains a  simplification of \eqref{histIndex2}, known as the functional varying coefficient model with history index \citep{sent:10}
\begin{align}
\label{histIndex3}
 X'(t) &=\alpha(t) +
 \beta_0(t)\int_0^{\tau_0}\zeta_0(u)X(t-u)du \nonumber \\ & + \sum_{j=1}^J\beta_{j}(t)\int_0^{\tau_{j}}\zeta_{j}(u)U_j(t-u)du+Z(t), \quad t \in [t_0,T] \nonumber\\ X(t) &= g(t), \quad t\in[t_0-\tau_0,t_0],
\end{align}
%where $t_0= \max_{0\leq j\leq J} \{\tau_j\}$.
%The initial condition process $g$ is continuous on $[t_0-\tau_0,t_0]$. 
where the history index function $\zeta_j(\cdot)$ in \eqref{histIndex3} defines the history index factor at $\beta_j(t),\ j = 0,\dots,J$ , by quantifying the influence of the recent history of the predictor values on the response. The varying coefficient function $\beta_j(t)$ represents the magnitude of this influence as a function of time. The functions $\zeta_j(\cdot), \beta_j(\cdot)$ and the intercept $\alpha(\cdot)$ are assumed to be smooth. For identifiability of model parameters, we assume that $\zeta_j(\cdot)$ is normalized by requiring that $\lVert\zeta_j\rVert_{L^2[0,\tau_j]}^2=\int_0^{\tau_j} \zeta^2_j(u) du = 1$ and, without loss of generality, $\zeta_j(0)>0$. Once $\zeta_j(\cdot)$ has been estimated, \eqref{histIndex3} reduces to a regular varying coefficient model. Due to its parsimonious nature and interpretability, model \eqref{histIndex3} has important applications in the biological and social sciences \citep{chio:12, hase:17}. 

\fi

For  COVID-19 modeling a hybrid model that combines distributed delay in $X(\cdot)$ with discrete delays for all other predictors $U_j(\cdot)$ turned out to be particularly suitable. In this practically relevant variant of a RDED one postulates that for  all $t \in [t_0,T]$,
\begin{align}
\label{eq:fcreg}
\case'(\tm) = \intc(\tm) + {\int_0^{\lag_{0}}\hislope(\tm,\vtm)\case(\tm-\vtm)\diffop\vtm} + \sum_{\pidx = 1}^{\nprdt} \slope_{\pidx}(\tm) \prdt_{\pidx}(\tm-\lag_{\pidx}) + Z(\tm),
\end{align}
{with an initial condition $X(t) = g(t), \ t\in[t_0-\tau_0,t_0].$}

We remark here that the RDED models proposed above are linear in the model parameters, although general RDED models determined by nonlinear functions are possible and might provide greater modeling flexibility, at the cost of additional modeling complexity. For example, it is not difficult to develop population models for quadratic and polynomial versions, in analogy to functional polynomial models \citep{mull:10:3}. Modeling with linear RDEDs  has the advantages that such models are easy to apply and implement, come with excellent interpretability of the model parameters, and exhibit good empirical performance in our application.  For the rest of the manuscript we therefore only consider linear RDEDs and leave the development application of  nonlinear RDEDs for 
future research.
%An assumption implicit in this model is that the history index function $\phi_j(\cdot)$ itself does not change over time, leading to a clear separation of time effects encoded in $\beta_j(\cdot)$ and history effects encoded in $\phi_j(\cdot)$, thus decomposing the functional regression of $Y$ on $X$ into these two easily interpretable one-dimensional component functions. \\

\section{Existence and Uniqueness of Solutions}
\label{sec: solutions}

\subsection{Functional concurrent model with discrete delays} 

\noindent Here we discuss the existence and uniqueness of the solutions of the RDEDs corresponding to the concurrent model \eqref{CONC:DEF} and the history-index functional linear model in \eqref{histIndex1}, as well as the hybrid model in \eqref{eq:fcreg} used in the data application.

%We use the Lebesgue space $\Lp$ as used in \cite{cala:19}. 
Consider a complete probability space $(\Omega,\mathcal{F},P)$, where $\mathcal{F}\subset2^{\Omega}$ 
is the $\sigma$-algebra on $\Omega$ and $P$ is a probability measure. Denote the space of random variables with finite $p^{\text{th}}$-moment by $\Lp$ $(p\geq 1)$, that is, for $Y: \Omega \rightarrow \mathbb{R},\ Y \in \Lp$ if $E|Y|^p<\infty$. A sequence of real-valued random variables $Y_m$ is said to converge to a random variable $Y$ in the $p^{\text{th}}$ moment if $||Y_m- Y||_p : = (E|Y_m- Y|^p)^{1/p}\rightarrow0$ as $m\rightarrow \infty$. We define a  process $Y(\cdot)$ to be $p^{\text{th}}$-moment continuous if for any $t_m \rightarrow t \in \mathcal{T}, \ ||Y(t_m) - Y(t)||_p \rightarrow 0 \text{ as } m \rightarrow \infty.$ The notions of $p^{\text{th}}$-moment differentiability  and $p^{\text{th}}$-moment Riemann integrability are defined similarly (see \cite{soon:73} Chapter $4$, pages $92$ and $100$).

We say that a stochastic process $X : \mathcal{T} \rightarrow \Lp$ is a solution of  an RDED {in the $\Lp$ sense} on the interval $\mathcal{T}$ if $X$ is  $p^{\text{th}}$-moment differentiable on $I$, $p^{\text{th}}$-moment continuous on $\mathcal{T}' : = [t_0-\tau_0,t_0]\cup \mathcal{T}$, and $X$ satisfies the RDED including the corresponding initial condition \citep{cala:19}. We require the following assumptions on the coefficient functions and predictor processes in model \eqref{CONC:DEF}.

\ben[label = (A\arabic*), series = freg, start = 0]
\item \label{ass:A1} $\alpha(\cdot$) and $\beta_0(\cdot)$ are continuous on $\mathcal{T} = [t_0,T]$.
\item \label{ass:A2} $Z(\cdot)$ is $p^{th}$-moment continuous on $\mathcal{T}$.
\item \label{ass:A3} The predictor processes $U_j$ satisfy $U_{j}(\cdot) \in \Lp$ and are $p^{\text{th}}$-moment continuous on $D_{j} : = [t_0 - \tau_{j}, T -\tau_{j}]$, for $j = 1,\dots, J.$
\een

\label{model:conc}
%\noindent With $f(X(t-\tau_0),t)=\alpha(t)+ \beta_0(t)X(t-\tau_0)+ \sum_{j=1}^J\beta_{j}(t)U_j(t-\tau_{j})+Z(t)$, the functional concurrent regression model for discrete delay RDED is given by,

%The existence and uniqueness of the concurrent model in \eqref{CONC:DEF} follows using arguments similar to the proof of Theorem 2.2 in \cite{cala:19}, which we rephrase in Theorem \ref{thm1}. A key step is to show that the function $f$ defined in \eqref{CONC:DEF} is $p^{\text{th}}$-moment continuous. All proofs are in Appendix A.1.\\

\cite{cala:19} showed the existence and uniqueness of the solution of  the following general form of a random delay differential equation with discrete delay $\tau>0$, given by
\begin{align}
    \label{calatayud}
    x'(t,\omega) &= r(x(t,\omega),x(t-\tau,\omega),t,\omega), \quad t \in [t_0,T],\nonumber\\
    x(t,\omega) &= r_0(t,\omega), \quad t_0-\tau\leq t\leq t_0
\end{align} and established  the following result. 
%\bthm[Theorem 2.2 \cite{cala:19}]

\noindent {\bf Theorem of Calatayud et al. (2019).} {\it If $r$ satisfies the  Lipschitz condition:
$\lVert r(x,y,t) - r(u,v,t)\rVert_p \leq k(t) \max \{\lVert x-u\rVert_p, \lVert y-v\rVert_p\},$
for $x,y \in \Lp,$ $t \in [t_0,T],$ and a $k$ with $\int_{t_0}^T |k(t)|dt <\infty,$ then the random delay differential equation \eqref{calatayud} has a unique solution in the $\Lp$ sense.}
%\ethm

The concurrent model in \eqref{CONC:DEF} is a special case of the random delay differential equation \eqref{calatayud} with a specific choice of $r$. A key step is to show that the function $f$ defined in \eqref{CONC:DEF} is $p^{\text{th}}$-moment continuous. The existence and uniqueness of  the solutions of \eqref{CONC:DEF} can then be obtained following similar arguments to those in the proof of Theorem 2.2 in \cite{cala:19}. We summarize this result in  Proposition \ref{thm1}. All proofs are in Appendix A.1.
%The concurrent model in \eqref{CONC:DEF} follows using arguments similar to the proof of Theorem 2.2 in \cite{cala:19}, which we rephrase in Theorem \ref{thm1}. A key step is to show that the function $f$ defined in \eqref{CONC:DEF} is $p^{\text{th}}$-moment continuous. All proofs are in Appendix A.1.\\
\begin{prop}
\label{thm1}
When $\tau_0>0$, under the regularity conditions \ref{ass:A1}-\ref{ass:A3} and assuming that the coefficients $\beta_{j}(\cdot)$ are continuous on $\mathcal{T} = [t_0,T],\ j = 1,\dots, J,$ the functional concurrent DDE model given by \eqref{CONC:DEF} has a unique solution in the $\Lp$ sense provided that the initial condition $g(\cdot)$ is $p^{\text{th}}$-moment continuous.
\end{prop}
\begin{rem}
\label{rem1} Proposition  \ref{thm1} requires the  $p^{\text{th}}$-moment continuity of $g(\cdot)$ on $s\in[t_0-\tau_0, t_0]$ for the uniqueness of the solution in the  $\Lp$ sense. If $g(\cdot)\in C^1([t_0-\tau_0, t_0], \mathbb{R})$ almost surely, then there exists a version of the $\Lp$-solution which solves the RDED in the sample path sense, see Theorem 2.3 in \cite{cala:19} and \cite{cara:19}. %For further conditions on $g(\cdot)$ to ensure regularity of the solution, the readers can refer to \cite{zhou2014stochastic}. 
Note that in \eqref{CONC:DEF}, we do not include the process $X(t)$ itself as one of the predictors; included is only the delayed version $X(t-\tau_0)$, where $\tau_0>0$. Hence solutions in the sample path sense do not require  $g$ to be differentiable almost surely. For  $\tau_0=0$, see Remark \ref{rem3}.
\end{rem}
\begin{rem} 
\label{rem2} 
The explicit sample path solution in the special case of a linear RDED has been derived  in \cite{cort:20, cara:19, cala:19}. The linear RDED often includes random coefficients and a random time-varying forcing term. In  \eqref{CONC:DEF}, the terms involving the predictor processes $U_j(\cdot)$, the drift process $Z(\cdot)$, and the intercept coefficient function $\alpha(\cdot)$ form the constituents of the random forcing term; the  varying coefficient functions $\alpha,  \beta_0, \beta_j$ are non-random. Using the method of steps as outlined in the proof of Theorem 2.3 in \cite{cala:19}, we can provide a solution  to \eqref{CONC:DEF} in the sample path sense as follows.
For $t \in [t_0, t_0 +\tau_0]$,  \eqref{CONC:DEF} is equivalent to an ordinary RDE given by
\begin{align}
\label{conc:sol1}
    &X'(t) = f(g(t-\tau_0),t)  \nonumber\\
    &X(t_0)  = g(t_0),
\end{align}
where $f$ is defined as in \eqref{CONC:DEF}. Then, the solution of  \eqref{conc:sol1} is given by
\begin{align*}
    X_1(t) := X(t)&= \int_{t_0}^{t} f(g(s-\tau_0),s) ds + g(t_0), \ t\in [t_0, t_0 +\tau_0].
\end{align*}
Next, for $t \in [t_0+\tau_0, t_0 +2\tau_0]$, 
\begin{align}
\label{conc:sol2}
    & X'(t) = f(X_1(t-\tau_0),t) \nonumber\\
    & X(t_0+\tau_0) = X_1(t_0 + \tau_0).
\end{align}
Solving \eqref{conc:sol2} leads to
\begin{align*}
    X_2(t) := X(t)&= \int_{t_0}^{t} f(X_1(s-\tau_0),s) ds + X_1(t_0+\tau_0), \ t\in [t_0+\tau_0, t_0 +2\tau_0].
\end{align*}
Repeating this argument until the entire domain $\mathcal{T}= [t_0,T]$ is covered %and for each segment, iterating the same logic, we get our desired 
one obtains the sample path solution. Due to the existence and uniqueness of the solution of  \eqref{CONC:DEF} in the $\Lp$ sense and by Remark \ref{rem1}, the sample path solution obtained as above is unique.
\end{rem}
\begin{rem} 
\label{rem3}
 In applications, the delay parameter can take the value zero and  when $\tau_0=0$  in \eqref{CONC:DEF} one arrives at an  ordinary RDE without delay. For this case,  existence and uniqueness of  solutions  has been well studied (see, e.g.,  Theorem 5.1.1 in \cite{soon:73} and \cite{stra:70}). %, which employ
%the continuity in $\Lp$ sense of $f$ in \eqref{CONC:DEF}. 
The sample path solution in this case can be obtained via the integration factor method and is given by
\begin{align*}
    X(t)&=e^{B(t)} \left( g(t_0)+ \int_{t_0}^t q(s) e^{-B(s)}ds \right), \\
    B(t)&:=\int_{t_0}^t \beta_0(u)du,
\end{align*}
where $q(t):=f(0,t)$ with $f$ defined as in \eqref{CONC:DEF}. Here $q(\cdot)$ consists of the intercept term $\alpha(t)$, the predictor processes $U_j(t)$ along with the corresponding coefficient functions $\beta_j(t),\ j = 1,\dots, J$, and the drift process $Z(t)$.
\end{rem}

\subsection{Functional linear model with history index for distributed delay}

\noindent We start with the simplest history index model with only one predictor process $U(\cdot)$, characterized as a first order RDED described in \eqref{histIndex1}. We prove the existence of a unique solution of \eqref{histIndex1} by adopting similar arguments as those  behind Proposition 2.1 and Theorem 2.2 of \cite{cala:19}. % modifying the argument as applicable in our context. 
The extension to multivariate predictor processes $\mathbf{U}(\cdot)$ is straightforward. Observe that one  can rewrite  model \eqref{histIndex1} as
\begin{align}
X'(t)
&= \alpha(t) + \ \int_{t-\tau_0}^{t}\gamma(t-s,t) X(s)ds + \int_{t-\tau_1}^{t} \gamma_1(t-s,t) U(s)ds +Z(t) \label{histIndex1:Calatayud1},
\end{align}
where $t \in \mathcal{T} = [t_0, T]$ with the initial condition $X(t)=g(t), \ t \in [t_0 - \tau_0,t_0].$

Assume that $g(\cdot)$ is {$p^{\text{th}}$-moment continuous} on $[t_0 - \tau_0,t_0]$ and define $\tilde{f}: \Lp \times I_b \rightarrow \mathbb{R}$ such that for {a $p^{\text{th}}$-moment continuous process} $Y(\cdot) : [t_0-\tau_0,b]\rightarrow \mathbb{R}$ and $t\in I_b := [t_0,b]$, $t_0<b,$
\begin{align}
\label{def:f:hist}
    \tilde{f}(Y,t)&= \alpha(t) + \ \int_{t-\tau_0}^{t}\gamma(t-s,t) Y(s)ds + \int_{t-\tau_1}^{t} \gamma_1(t-s,t) U(s)ds +Z(t).
\end{align}
Here $\tilde{f}$ is a random functional whose first argument is  the trajectory $\{Y(s) : s \in [t_0-\tau_0,b]\}$ and whose  second argument  is $t \in [t_0,b]$. The following proposition with proof in Appendix A.2 provides the continuity of $\tilde{f}$ in $t$ in the $\Lp$ sense.
\begin{prop}
\label{prop:cont:hist}
Suppose $Y:[t_0-\tau_0,b]\rightarrow \Lp $ is $p^{\text{th}}$-moment continuous for $\tau_0 >0$. Then $\tilde{f}$ as defined in \eqref{def:f:hist} is continuous in $t$ in the $\Lp$ sense on the domain $[t_0,b],\ t_0<b$. 
\end{prop}
%\noindent The proof of the continuity of $\tilde{f}(Y,t)$, and hence the $p^{\text{th}}$- moment continuity of the process $\tilde{f}(Y,\cdot)$, can be found in the Appendix B.
For the random process $X$ on $[t_0-\tau_0, T]$ and the time variable $t \in [t_0,T]$,  the history index model \eqref{histIndex1:Calatayud1} can be expressed as
 \begin{align}
 \label{histIndex1:Calatayud2}
 X'(t) &=\tilde{f}(X,t).
 \end{align}
For the existence of a unique solution of \eqref{histIndex1:Calatayud2} we need the following assumptions, in addition to \ref{ass:A1}- \ref{ass:A3}:
\ben[label = (B\arabic*), series = freg, start = 0]
     \item \label{ass:B1} $\gamma$ is continuous on $\mathcal{T}'\times \mathcal{T} = [t_0-\tau_0,T]\times [t_0,T]$ and  $\gamma_1$ is continuous on $\mathcal{T}'' \times \mathcal{T} = [t_0-\tau_1,T] \times [t_0,T]$.
     \item \label{ass:B2} $X(\cdot)\in \Lp$ on $\mathcal{T}'= [t_0-\tau_0,T]$ and is $p^{\text{th}}$-moment continuous on $I'$.
\een 
The following result 
characterizes the solution of \eqref{histIndex1:Calatayud2}, where the integral in condition~\ref{cond:b} of the following Proposition is defined in the $p^{\text{th}}$-moment Riemann integral sense \cite[for definition of $p^{\text{th}}$-moment Riemann integrability see][Chapter $4$ page $100$]{soon:73}.
\begin{prop}
\label{prop:Lpsol:hist}
The process  $X: \mathcal{T} \rightarrow \Lp$ is a  $\Lp$ solution of \eqref{histIndex1:Calatayud2} if and only if for $\tau_0>0,$
\ben[label=(\alph*)]
    \item \label{cond:a} $X$ is $p^{\text{th}}$-moment continuous on $[t_0-\tau_0 ,T]$
    \item \label{cond:b} $X(t) = g(t_0) + (\Lp) \int_{t_0}^{t} \tilde{f}(X,s)ds$ for each $t \in \mathcal{T}$% the integral being defined in the $p^{\text{th}}$-moment Riemann integral sense \cite[for definition of $p^{\text{th}}$-moment Riemann integrability see][Chapter $4$ page $100$]{soon:73}.
    \item \label{cond:c} $X(t) = g(t)$ for each $t\in [t_0-\tau_0,t_0]$.
\een
\end{prop}

The following result establishes   existence and uniqueness of a solution for \eqref{histIndex1:Calatayud2}.  The proof extends the classical Picard theorem for deterministic ODE to RDED in the  $\Lp$ sense, via the Banach fixed-point theorem, following a similar line of arguments as in \cite{cala:19} (see Appendix A.2). 
\bthm
\label{thm2}
Assume the regularity conditions \ref{ass:A1}-\ref{ass:A3} and \ref{ass:B1}-\ref{ass:B2}. Then for $\tau_0>0$ we have that 
$$k(t)=\int_{t-\tau_0}^t|\gamma(t-s,t)|ds\in L^1(\mathcal{T}), \ t\in \mathcal{T} = [t_0,T]$$
and a solution of \eqref{histIndex1} exists and is unique in the $\Lp$ sense.
\ethm
%We remark that the extension of Theorem \ref{thm2} to the multivariate case \eqref{histIndex2} 
%and the hybrid model in \eqref{eq:fcreg}
%is straightforward under suitable regularity conditions. 
For the existence and uniqueness of the solution of the hybrid model in \eqref{eq:fcreg}, which is the most promising of the models considered  for modeling the growth rate of COVID-19, the following corollary applies. 
\begin{cor}
\label{cor1}
Suppose the regularity conditions \ref{ass:A2}, \ref{ass:A3} and \ref{ass:B2} hold. If $\alpha(\cdot)$ and the coefficients $\beta_{j}(\cdot)$ are continuous on $ \mathcal{T}= [t_0,T],\ j = 1,\dots, J,$ and $\gamma$ is continuous on $\mathcal{T}'\times \mathcal{T} = [t_0-\tau_0,T]\times [t_0,T]$, then the model in \eqref{eq:fcreg} has a unique $\Lp$ solution.
\end{cor}

\section{Learning Dynamics from Samples of Multivariate Stochastic  Processes and Practical Considerations for COVID-19 Caseload Modeling}\label{sec:fitting}

\subsection{Obtaining derivatives}
\label{sec:derivest}

\noindent In data-driven differential equation modeling, given a sample of observed processes, a necessary first step is the estimation of  derivatives. In applications  such as the COVID-19 case trajectories the available data are usually noise-contaminated, and therefore some care is needed to obtain viable derivative estimates.
The situation for the observed data for one of the trajectories is reflected by the nonparametric regression model 
\begin{equation}
    Y_k = X(t_k) + \epsilon_k, \quad k = 1,...,K,
    \label{nonparareg}
\end{equation}
%for each state's trajectory is useful to describe the generation of the observations. 
where $\{t_k\}_{k=1}^K$ are the grid points where the process $X(t)$ is measured and $\epsilon_k$ are the measurement errors, which typically are considered to be independent mean zero random variables with finite variance. 

In our data application, the available data $Y_k$ correspond to  the cumulative case counts of COVID 19, which we consider to be noisy realizations of a smooth underlying function $X(\cdot)$ that is observed on a daily grid. %a time grid $\{t_k\}_{k=1}^K$ corresponding to each day available in the COVID-19 data, contaminated by the errors $\epsilon_k$. For the task of estimating the 
For the practical estimation of the derivatives $X'(t)$ one has various choices that all require a tuning parameter. Since in preliminary studies local polynomial fitting yielded the smallest leave-one-out prediction error relative to the difference quotients, which serve as nearly unbiased targets, we opted to obtain derivatives by applying local polynomial regression \citep{mull:87, fan:96} as per Section A.3 of the Appendix, where further details are provided.

\subsection{Target model and initial lag and covariate process selection}\label{sec:varSelect}

\bco

For learning empirical dynamics of the COVID-19 case trajectories given covariate processes $U_j$,  we adopt a two stage procedure to estimate varying coefficient functions and  time delays. We first fit model  \eqref{eq:history-nocov} to the derivative of the process driving the RDED to estimate the historical weight surface $\gamma(\cdot,\cdot)$ given delay $\tau_0$ and then integrate the driving process weighed by the estimated history index surface across the delay period. The integral is used as a covariate in a functional concurrent regression model along with the covariate processes $U_j$ that exhibit  individual delay components. %In what follows, we discuss the estimation steps for the model parameters and the covariate delays. We also discuss  variable selection for the concurrently observed covariate processes $\mathbf{U}(\cdot)$.

\fi

%In order to study the delay time dynamics of \COVID in the United States, we focus on the state-specific time-varying cumulative confirmed cases of \COVID per million people, $\case(\tm)$, for $\tm\in\tdom$, which is referred to as the \emph{case process} in the following. 
%Here, the time domain considered $\tdom$ is from April 5 to August 12 in the year of 2020, where data are recorded once per day, for $\{\tm_{\tidx}\}_{\tidx=1}^{\ntgrid}$, with $\ntgrid = 129$. 
%We adopt the functional varying-coefficient regression in functional data analysis to study the relationship between the derivative of the case process and the case process itself together with $\nprdt=34$ other variables as listed in \red{Table \ref{table:variables} (Could data team please add the table label here?)}, 

For modeling the COVID-19 caseload process $X$ for the states of the U.S. 
%the process of interest $\case(\tm)$ with $\tm\in\tdom$ , for which we have in mind the COVID case trajectories for the states within the U.S., 
we adopt the RDED model   \eqref{eq:fcreg},
\begin{equation*}
 \case'(\tm) = \intc(\tm) + {\int_0^{\lag_{0}}\hislope(\tm,\vtm)\case(\tm-\vtm)\diffop\vtm} + \sum_{\pidx = 1}^{\nprdt} \slope_{\pidx}(\tm) \prdt_{\pidx}(\tm-\lag_{\pidx}) + Z(\tm), \quad \tm\in\tdom
 \end{equation*}
{where the time delay of the effect of the process $\case$ at time $\tm\in\tdom$ is distributed on the interval $[\tm-\lag_{0},\tm]$}, and $\{\lag_{\pidx}\}_{\pidx=1}^{\nprdt}$ are  discrete time delays or time lags for the covariate or predictor processes  $\prdt_{\pidx}$.  %As before, the error term $Z(t)$ denotes a mean-zero  drift process. %We denote the process and other time-varying predictors for state $\sidx$ by $\case_{\sidx}$ and $\prdt_{\sidx\pidx}$, respectively. 
Our starting point is that the data  $\{(\case_{\sidx},\prdt_{\sidx1}, \dots, \prdt_{\sidx\nprdt})\}_{\sidx=1}^{\nsub}$ represent  $\nsub$ independent realizations of the stochastic process $(\case,\prdt_{1}, \dots, \prdt_{\nprdt})$, where each realization (corresponding to a state)  is indexed by $i$ and  $n$ is the number of states included in the analysis. 
%In practice, such processes are usually not fully observed. Rather,  the available observations are measurements on a discrete time grid  $\{\tm_{\tidx}\}_{\tidx=1}^{\ntgrid}$ as in (\ref{nonparareg}).

Given $\tau_0$, to estimate the history index weight function $\hislope(s, t)$ in the historical model without covariates, we preform a scalar-to-function linear regression for each grid point $t_k$, i.e.,
\bal
\label{eq:history-nocov}
X'(t_k) &=  \alpha(t_k) +\int_0^{\tau_0}\gamma(s,t_k) X(t_k-s)\ ds +Z(t_k), \,\,\, 1 \le k \le K.
\eal
This can be implemented using function \texttt{FLM} in the R package \texttt{fdapace} \citep{fdapace}. With estimates of $\gamma(s, t_k)$ for each $k$ in hand,  the estimate of the history index weight function, $\ehislope(s, t)$ for all $s,t\in\tdom$, is obtained using a two dimensional kernel smoother.

%This can be implemented using function \texttt{FLM} in the R package \texttt{fdapace} \citep{fdapace}. }
For each of the covariates  $\prdt_{\pidx}$, we conduct an initial lag selection based on functional varying-coefficient regression including the corresponding variable as a single predictor, considering the models 
%Specifically, for each $\pidx=0,\dots,\nprdt$, we consider the model 
\bgt\label{eq:fcreg_singlePrdt}
\case'(\tm_{\tidx}) = \intc_{\pidx}(\tm_{\tidx}) + \vslope_{\pidx}(\tm_{\tidx})\prdt_{\pidx}(\tm_{\tidx} -\lag)+\err_{\pidx\tidx}, \,\,\, 0 \le j \le J, \,\, 1 \le k \le K,\egt
where $\prdt_{0} \coloneqq \case$. 
Given a lag $\lag$, for each subject $\sidx$ and time point $\tm_{\tidx}$, we obtain the leave-one-out prediction 
\bgt\nn
\elag\case'_{\sidx,-\sidx}(\tm_{\tidx};\lag) = \elag\intc_{\pidx,-\sidx}(\tm_{\tidx}) + \elagVslope_{\pidx,-\sidx}(\tm_{\tidx})\prdt_{\sidx\pidx}(\tm_{\tidx}-\lag), \egt
where $\elag\intc_{\pidx,-\sidx}(\tm_{\tidx})$ and $\elagVslope_{\pidx,-\sidx}(\tm_{\tidx})$ are estimated by fitting model \eqref{eq:fcreg_singlePrdt} by ordinary least squares,  excluding the data from subject $\sidx$. 
The lags are chosen to minimize the leave-one-out prediction error across subjects, 
\bgt\label{eq:inilag}
\elag\lag_{\pidx} = \argmin_{\lag\in\{0,1,\dots,21\}} \frac{1}{\nsub\ntgrid} \sum_{\sidx = 1}^{\nsub} \sum_{\tidx=1}^{\ntgrid} \left(\case'_{\sidx}(\tm_{\tidx}) - \elag\case'_{\sidx,-\sidx}(\tm_{\tidx};\lag)\right)^2. \egt

With initial choices of the lags $\{\elag\lag_{\pidx}\}_{j=1}^J$ {and the estimate $\ehislope$ obtained from model \eqref{eq:history-nocov}}, for each $\tm_{\tidx}$, 
variable selection is then performed using LASSO \citep{tibs:96},
\bgt\nn%\label{eq:lasso}
\left(\evs\intc(\tm_{\tidx}),(\evs\slope_{\pidx}(\tm_{\tidx}))_{\pidx=0}^{\nprdt}\right) %\\&
= \argmin_{\intc_{\tidx},\bm\slope_{\tidx}}  Q(\intc_{\tidx},\bm\slope_{\tidx}),\\
\aligned
Q(\intc_{\tidx},\bm\slope_{\tidx}) 
&=  \sum_{\sidx = 1}^{\nsub} \left(\case'_{\sidx}(\tm_{\tidx}) - \intc_{\tidx} - \slope_{0\tidx} {\int_0^{\lag_{0}}\ehislope(\tm_{\tidx},\vtm)\case(\tm_{\tidx}-\vtm)\diffop\vtm} - \sum_{\pidx = 1}^{\nprdt} \slope_{\pidx\tidx} \prdt_{\sidx\pidx}(\tm_{\tidx}-\elag\lag_{\pidx})\right)^2\\
&\quad + \pnty \sum_{\pidx = 0}^{\nprdt}|\slope_{\pidx\tidx}|,
\endaligned
\egt
where $\bm\slope_{\tidx} = (\slope_{0\tidx},\dots,\slope_{\nprdt\tidx})$. 
This can be implemented using the R package \texttt{ncvreg} \citep{ncvreg, ncvreg-pkg}, where the tuning parameter $\pnty$ is chosen by leave-one-out cross-validation. 
Considering all $\ntgrid$ time points simultaneously, we obtain proportions of the number of time points $\tm_{\tidx}$ where each variable $\prdt_{\pidx}$ is selected, i.e., 
\bgt\nn \psel_{\pidx} = \#\{\tidx: \evs\slope_{\pidx}(\tm_{\tidx}) \ne 0\} / \ntgrid,\quad\text{for } \pidx=1,\dots,\nprdt.\egt 

Subsequently, selection among the time-varying predictors $\{\prdt_{\pidx}\}_{\pidx=1}^{\nprdt}$
other than the process $\case$  is performed by applying a threshold on these proportions $\{\psel_{\pidx}\}_{\pidx=1}^{\nprdt}$. This  threshold {$\psel^*$} is chosen by leave-one-out cross-validation, 
\bgt\label{eq:varSel} \{\prdt_{\pidx}: \psel_{\pidx} \ge {\psel^*}\}.\egt
%specifically, the selected variables are listed in Table~\ref{table:lags}.

\subsection{Lag selection by backfitting}\label{sec:backfitting}
 
Let $\mathcal{V}$ denote the set of selected predictor processes  obtained using the method described in Section~\ref{sec:varSelect}.  Updated lag parameters are then chosen for the variables in $\mathcal{V}$ through an iterative backfitting algorithm such that the total leave-one-out prediction error is minimized. Given the $b^{th}$ iteration of the lag vector, one obtains the $(b+1)^{st}$ update for the $k^{th}$ lag by minimizing the leave-one-out integrated mean squared prediction error, i.e.,

$$ \lag_k^{(b+1)} = \argmin_{\lag_k}\frac{1}{n}\sum_{i=1}^{n}\int_\mathcal{T}
(X'_i(t) - \hat{X}_{i,-i}'^{(b)}(t;\lag_k))^2 dt,$$
\noindent where for $k \in \mathcal{V}$
\begin{align*}
    \hat{X}_{i,-i}'^{(b)}(t;\lag_k) =&~ \hat{\alpha}^{(b)}_{-i}(t)+ {\hat{\beta}^{(b)}_{0,-i}(t)\int_0^{\lag_{0}}\ehislope(\tm,\vtm)\case_i(\tm-\vtm+\lag_{0})\diffop\vtm}\\&+\sum_{j\neq k, j \in \mathcal{V}}
     \hat{\beta}^{(b)}_{j,-i}(t)U_{j,i}(t-\lag_j^{(b)})+ \hat{\beta}^{(b)}_{k,-i}(t)U_{k,i}(t-\lag_k).
\end{align*}

Here the functions $\hat{\beta}^{(b)}_{k,-i}(t)$ represent the fitted coefficients corresponding to the $b^{th}$ iteration of the $j^{th}$ lag, $j \in \mathcal{V}$ with the $i^{th}$ observation held out for prediction. This process is iterated over $k \in \mathcal{V}$ until the lag vector converges. Lags are initialized at the vector $\bm{\lag}^{(0)} = (\elag\lag_{1},\dots,\elag\lag_{j},\dots, \elag\lag_{|\mathcal{V}|})^T$, where $\elag\lag_j$ are as described in Section~\ref{sec:varSelect} and the stopping rule is such that the algorithm terminates if no lags have changed after a complete cycle of updates for all predictors.  
After the final lags are estimated, we lightly smooth the estimated coefficients with local linear smoothing, which is a standard procedure in concurrent functional regression modeling \citep{fan:00:2,mull:08:4}. Specifically, the initial estimates $\hat{\beta}_j(t_k)$ are smoothed to obtain the effect functions presented in Figures 3 and 4, defined as $\hat{\beta}^{smooth}_j(t) = \hat{b}^{(0)}_j$ where 
\begin{equation}
(\hat{b}^{(0)}_{j},\hat{b}^{(1)}_{j}) = \argmin_{b^{(0)}_{j},b^{(1)}_{j}} \sum_{k=1}^{K}K\left(\frac{t_{k}-t}{h}\right)(\hat{\beta}_j(t_k)-b^{(0)}_{j}-b^{(1)}_{j}(t-t_{k}))^2,
\label{eq:lwls}
\end{equation}
with $K(\cdot)$ representing the Epanechnikov kernel with a bandwidth of $h=20$ days.

%\subsection{Bootstrap confidence intervals for concurrent effect functions}

 %A non-parametric bootstrap procedure is performed in order to obtain the pointwise confidence bands for the regression coefficient functions $\beta_j(\cdot)$. Specifically, the subjects are resampled with replacement and the resulting regression coefficient estimates form the bootstrap replicates which are used to produce the bootstrap confidence bands. In the data application, we repeat this procedure $B=10, 000$ times. The lower and upper bound are then obtained by taking  2.5\% and 97.5\% quantiles of the $B$ estimates pointwisely.

\section{Random Differential Equation with Delay Fitting and Prediction for COVID 19 Case Trajectories in the United States}

\subsection{Data Description}

We obtained daily confirmed cases across states in the United States from the COVID-19 Data Repository by the Center for Systems Science and Engineering (CSSE) at Johns Hopkins University, which is publicly available at \url{https://github.com/CSSEGISandData/COVID-19} and was accessed on September $19$, $2020$. 
The positive and negative COVID-19 test results per day and state were obtained from the COVID Tracking Project and are  publicly available at \url{https://COVIDtracking.com/} (accessed on September $19$, $2020$). The former variable refers to the number of people with confirmed or probable case of COVID-19 (see \url{https://COVIDtracking.com/about-data/data-definitions} for more details) and the latter refers to the total number of people with a completed negative PCR test. We used effective positivity rates (EPR), defined as the ratio of effective positive results (positive test or probable case) to the effective total tests (positive or probable case plus negative test results).

%, and similarly the variable `posratio' which corresponds to the ratio of effective positive results to the state population in $2019$. %We also consider the population density which was constructed using the state area information available at \url{https://www.census.gov/geographies/reference-files/2010/geo/state-area.html} and the $2019$ population for each state, which is available from the Economic Tracker databases as introduced next. 

\bco

\begin{longtable}{|r|l|}
    
\hline \hline
Item & Feature description\\
\hline \hline
1 & cases per million\\
\hline
2 & ratio of effective positive results to effective total tests\\
\hline 
3 & ratio of effective positive results to state population in $2019$\\
\hline
4 & individuals who are currently hospitalized with COVID-19\\
\hline
5 & spending in all categories \\
\hline
6 & spending in arts, entertainment, and recreation categories \\
\hline
7 & spending in accommodation and food service categories\\
\hline
8 & spending in general merchandise stores and apparel and accessories categories\\
\hline
9 & spending in grocery and food store categories\\
\hline
10 & spending in health care and social assistance categories\\
\hline
11 & spending in transportation and warehousing categories\\
\hline
12 & spending among high (top quartile) income ZIP codes in all categories\\
\hline
13 & spending among low (bottom quartiles) income ZIP codes in all categories\\
\hline
14 & spending among middle (middle two quartiles) income ZIP codes in all categories\\
\hline
15 & time spent at retail and recreation locations\\
\hline
16 & time spent at grocery and pharmacy locations\\
\hline
17 & time spent at parks\\
\hline
18 & time spent at inside transit stations\\
\hline
19 & time spent at work places\\
\hline
20 & time spent at residential locations\\
\hline
21 & time spent outside of residential locations\\
\hline
22 & number of small businesses open\\
\hline
23 & number of small businesses open among high(top quartile) income ZIP codes\\
\hline
24 & number of small businesses open among low(bottom quartile) income ZIP codes\\
\hline
25 & number of small businesses open among middle(middle two quartiles) income ZIP codes\\
\hline
26 & number of small businesses open in transportation\\
\hline
27 & number of small businesses open in education and health services\\
\hline
28 & number of small businesses open in leisure and hospitality\\
\hline
29 & net revenue for small businesses\\
\hline
30 & net revenue for small businesses among high(top quartile) income ZIP codes\\
\hline
31 & net revenue for small businesses among low(bottom quartile) income ZIP codes\\
\hline
32 & net revenue for small businesses among middle(middle two quartiles) income ZIP codes\\
\hline
33 & net revenue for small businesses in transportation\\
\hline
34 & net revenue for small businesses in education and health services\\
\hline
35 & net revenue for small businesses in leisure and hospitality\\
\hline
\caption{List of all variables considered (see Data Description Section for further details).}
\label{table:variables}
\end{longtable} 

\fi

We also obtained features from the \emph{Opportunity Insights Economic Tracker Data} \citep{economicTracker}, which is publicly available at \url{https://github.com/OpportunityInsights/EconomicTracker} (accessed on September $19$, $2020$). This database contains daily and state level information for several economic activity indicators such as credit and debit card expenditure across different types of activities, some of them stratified by  zip codes with low/medium/high income level. We refer to \cite{economicTracker} for more details. Additionally, we  downloaded Google mobility data, publicly available at \url{https://www.google.com/COVID19/mobility/} (accessed on September $19$, $2020$). These data include indicators of  mobility pattern changes in different areas such as parks, residential locations, retail locations, among others; and information such as the percent change in revenue as well as total number of small businesses for high/low income consumers and across different economic activities; see Table \ref{table:variables} for a complete listing.

{\tiny

\begin{longtable}{|r||l|}
\hline \hline
Item & Feature description\\
\hline \hline
1 & cases per million\\
\hline
2 & ratio of effective positive results to effective total tests\\
\hline 
3 & ratio of effective positive results to state population in $2019$\\
\hline
4 & individuals who are currently hospitalized with COVID-19\\
\hline
5 & spending in all categories \\
\hline
6 & spending in arts, entertainment, and recreation categories \\
\hline
7 & spending in accommodation and food service categories\\
\hline
8 & spending in general merchandise stores and apparel and accessories categories\\
\hline
9 & spending in grocery and food store categories\\
\hline
10 & spending in health care and social assistance categories\\
\hline
11 & spending in transportation and warehousing categories\\
\hline
12 & spending among high (top quartile) income ZIP codes in all categories\\
\hline
13 & spending among low (bottom quartiles) income ZIP codes in all categories\\
\hline
14 & spending among middle (middle two quartiles) income ZIP codes in all categories\\
\hline
15 & time spent at retail and recreation locations\\
\hline
16 & time spent at grocery and pharmacy locations\\
\hline
17 & time spent at parks\\
\hline
18 & time spent at inside transit stations\\
\hline
19 & time spent at work places\\
\hline
20 & time spent at residential locations\\
\hline
21 & time spent outside of residential locations\\
\hline
22 & number of small businesses open\\
\hline
23 & number of small businesses open among high(top quartile) income ZIP codes\\
\hline
24 & number of small businesses open among low(bottom quartile) income ZIP codes\\
\hline
25 & number of small businesses open among middle(middle two quartiles) income ZIP codes\\
\hline
26 & number of small businesses open in transportation\\
\hline
27 & number of small businesses open in education and health services\\
\hline
28 & number of small businesses open in leisure and hospitality\\
\hline
29 & net revenue for small businesses\\
\hline
30 & net revenue for small businesses among high(top quartile) income ZIP codes\\
\hline
31 & net revenue for small businesses among low(bottom quartile) income ZIP codes\\
\hline
32 & net revenue for small businesses among middle(middle two quartiles) income ZIP codes\\
\hline
33 & net revenue for small businesses in transportation\\
\hline
34 & net revenue for small businesses in education and health services\\
\hline
35 & net revenue for small businesses in leisure and hospitality\\
\hline
\caption{List of all variables considered (see Data Description Section for further details).}
\label{table:variables}
\end{longtable} 
}

We briefly describe a few features below and  refer to the Economic Tracker source \citep{economicTracker} for further details:
\begin{itemize}
    \item `\text{spend all}': Seasonally adjusted credit and debit card spending relative to the period Jan $4$ to Jan $31$, $2020$, in all merchant categories, reported  as a seven day moving average formed by using the values for the current and the previous six days.
    \item `\text{spend tws}': Seasonally adjusted credit and debit card spending relative to the period Jan $4$ to Jan $31$, $2020$, in the category transportation and warehousing, and also reported as a seven day moving average.
    \item `\text{spend all inchigh}': Seasonally adjusted credit and debit card spending by consumers living in high median income (above $75\%$ quantile) zip codes, and relative to the period Jan $4$ to Jan $31$, $2020$, in all merchant categories, again reported as a seven day moving average.
    \item `\text{revenue all}': Seasonally adjusted percent change in net revenue for small businesses and indexed to the period Jan $4$ to Jan $31$, $2020$, reported as a seven day moving average.
    \item `\text{revenue ss70}': Seasonally adjusted percent change in net revenue for small businesses in leisure and hospitality and indexed to the period Jan $4$ to Jan $31$, $2020$, reported as a  seven day moving average.
    \item `\text{merchant inchigh}': Seasonally adjusted percent change in the number of small businesses open in high median income (above $75\%$ quantile) zip codes and indexed to the period Jan $4$ to Jan $31$, $2020$,  reported as a seven day moving average.
    \item `gps parks`: Corresponds to the time spent at parks relative to a baseline which is defined as the median value, for the corresponding day of the week, during the period Jan $3$--Feb $6$, $2020$.
\end{itemize}

%The previous features are already presmoothed  by using a $7$-day moving average from the current and previous six days. 
We employ local linear regression %\citep{fan:96} 
to slightly smooth the feature trajectories; this also  allows to impute missing observations for some states and days. % as well as to obtain a smoothed representation of the predictor processes. We used 
For this step, a  Gaussian kernel with bandwidth $1.5$ days was employed  and we utilized the R package \texttt{fdapace} version 0.5.5 for the computational implementation \citep{fdapace}. The effective test positivity rate was smoothed analogously. We use local quadratic regression \citep{fan:96} to obtain the derivatives $\case'(\tm)$, implemented by using the R package \texttt{locfit} version 1.5-9.4 \citep{locfit}. For data analysis we consider states with population at least 1 million as the states with smaller populations had less COVID-19 cases, and their inclusion would therefore increase the overall noise level in the data and negatively impact the analysis. This is a pragmatic choice and is not a limitation of the proposed method.  %and rather an application oriented interest to select states with significant growth trends in the number of COVID-19 cases. 
In the end, we included the data from 44 states (excluding Alaska, Delaware, North and South Dakota, Vermont and Wyoming). 

%For the study we only consider states with a population above $1$ million, resulting in  $44$ included states  (excluding  %The six states that were not considered are then: 
%Alaska, Delaware, North and South Dakota, Vermont and %Wyoming). 

%The features `posratio' and `posToTotalTests' were smoothed analogously.

As a result of the LASSO variable selection described in Section 4.3,
which was implemented with the R package \texttt{ncvreg} version 3.12.0 \citep{ncvreg-pkg} with threshold  $\psel^*=0.3$ in \eqref{eq:varSel}, the features utilized in the final model are as follows, where the selected lags are listed in Table \ref{table:lags}.
\begin{enumerate}
    \item $X_{i}(t)$, the total cases per million at time $t$ for state $i$, $i=1,\dots,44$.
    \item $U_{1i}(t)$, the relative time spent at park areas at time $t$ in state $i$.
    \item $U_{2i}(t)$, the number of hospitalized people at time $t$ in state $i$.
    \item $U_{3i}(t)$, the ``positive testing rate" defined as the ratio of effective positive results (positive test or probable case) to the effective total tests (positive or probable case plus negative test results) at time $t$ for state $i$.
    \item $U_{4i}(t)$, the credit/debit card relative spending in the arts, entertainment, and recreation category (see `spend aer' predictor in Table~\ref{table:variables}) at time $t$ in state $i$.
\end{enumerate}

\begin{table}[H]
\centering
 \begin{tabular}{c|c| ccc } 
 \hline
 predictor & variable && lag\\ [0.5ex] 
 \hline\hline
 
  park mobility activity & $U_1(t)$ && 0 &\\
 \hline
 
  number of people currently hospitalized & $U_2(t)$&& 14 &\\
 \hline

effective test positivity rate & $U_3(t)$ && 7 &\\
 \hline

 \makecell{credit/debit card spending on\\ arts, entertainment, and recreation}& $U_4(t)$ & &3 &\\
 \hline
\hline
\end{tabular}
 \caption{Predictors and their corresponding lags (see Data Description Section for further details).}
 \label{table:lags}
\end{table}

\subsection{Model Fitting and Results}

In order to study the delay time dynamics of COVID-19 in the United States, we focus on the state-specific time-varying cumulative confirmed cases per million people, $\case(\tm)$, for $\tm\in\tdom$, which is referred to as the \emph{case process} in the following. 
Here, the time domain $\tdom$ that we consider starts on April 5 and ends on  August 12 in  2020, where data are recorded once per day, on an equidistant grid for 130 days. %for $\{\tm_{\tidx}\}_{\tidx=1}^{\ntgrid}$, with $\ntgrid = {130}$. 
{We estimate the history coefficient function $\hislope$ in the RDED model  \eqref{eq:fcreg} as described in Section \ref{sec:varSelect} with $\lag_{0} = 14$, and} adopt the functional varying-coefficient regression in model \eqref{eq:fcreg} with $\nprdt=34$ other variables as listed in Table~\ref{table:variables}. {The estimated coefficient surface $\hat{\gamma}(s, t)$ and some of its cross-sections  are displayed in Figure~\ref{fig:surface}. % Final model predictors obtained after implementing variable and lag selection procedures described in Sections~\ref{sec:varSelect} and \ref{sec:backfitting} are listed alongside their optimal lags in Table~\ref{table:lags}Concurrent effect functions for the intercept and all regressors are displayed in Figure~\ref{fig:beta}. %A stretch of time where the confidence band does not touch 0 (colored in red) suggests that the effect of the predictor is locally significant during that time interval.

For each fixed $t$, $\hat{\gamma}(\cdot, t)$, the estimated history index in model \eqref{eq:fcreg} is found to change  from positive near current time to negative with increasing lags. This means that the recent past of the process has a positive association with the current derivative while days further into the past are negatively associated. The right panel shows that the shape of $\hat{\gamma}(\cdot, t)$ changes from more or less quadratic to linear from the early days of COVID-19 to recent times, indicating that in April 2020 the contrast between the effects of near-current and past caseload  on the current derivative was more pronounced than it is more recently. 
%\begin{figure}[htbp]
%	\centering
%	\label{fig:surface}
%	\includegraphics[width=\linewidth]{finalplots/gamma_surface.pdf}
%	\caption{Estimated coefficient surface $\hat{\gamma}(s, t)$, where $t$ corresponds to number of days since 2020-04-05 and $s$ takes values from $0$ to $\tau_0$.}
%\end{figure}

\begin{figure}[htbp]
    \centering
    \subfloat[]{
    \centering
    \includegraphics[width=0.5\textwidth]{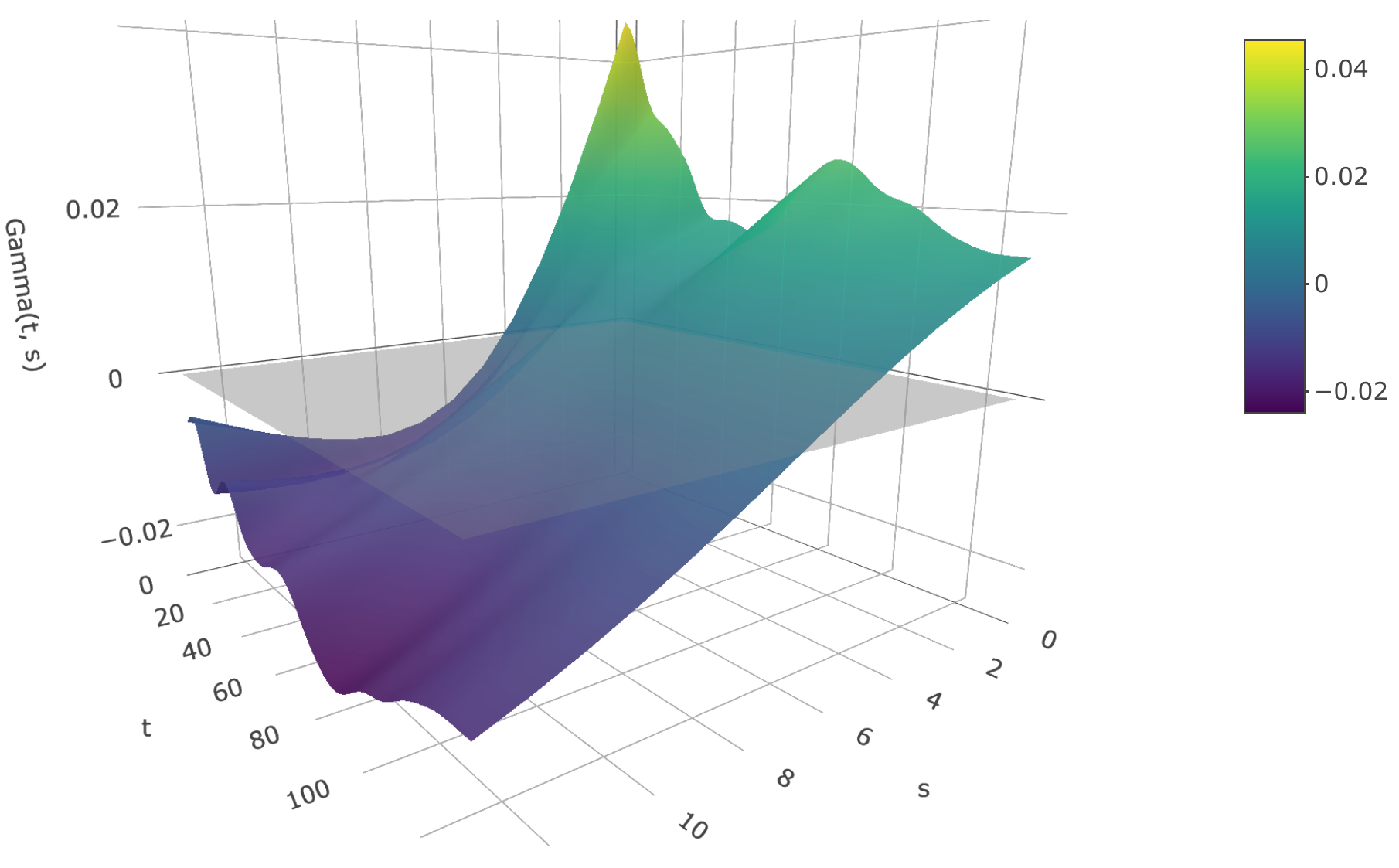}
    }
    \subfloat[]{
    \centering
    \includegraphics[width=0.5\textwidth]{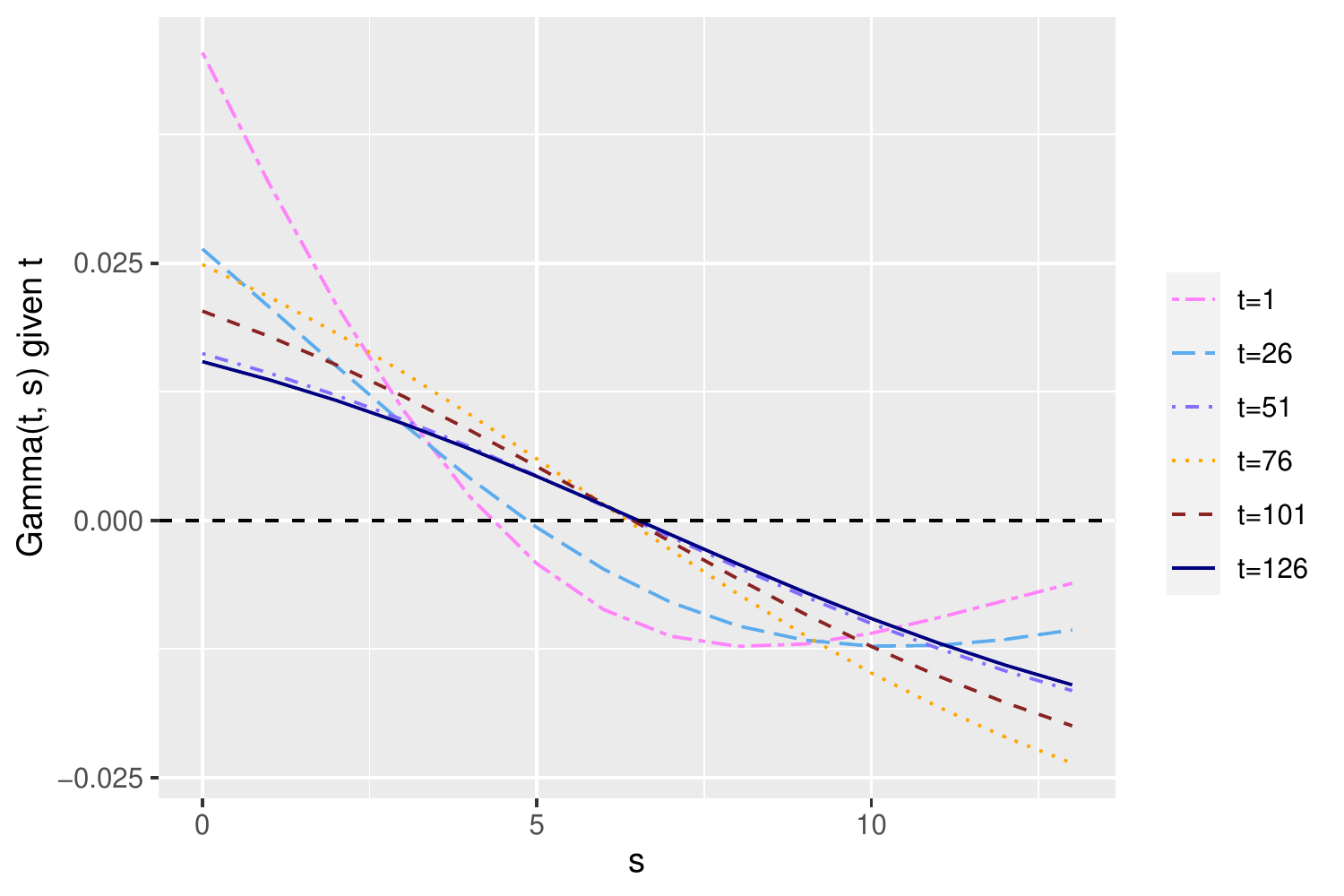}
    }
\caption{(a) Estimated coefficient surface $\hat{\gamma}(s, t)$ in the RDED  model  \eqref{eq:fcreg}, 
where $t$ corresponds to number of days since 2020-04-05 and $s$ takes values from $0$ to $\tau_0$.
(b) $t$-sliced cross Sections for  $t = 1, 26, 51, 76, 101, 126$ of $\hat{\gamma}(s, t)$ as functions of  $s$.}
\label{fig:surface}
\end{figure}

\begin{figure}[htbp]
	\centering
	\includegraphics[width=\linewidth]{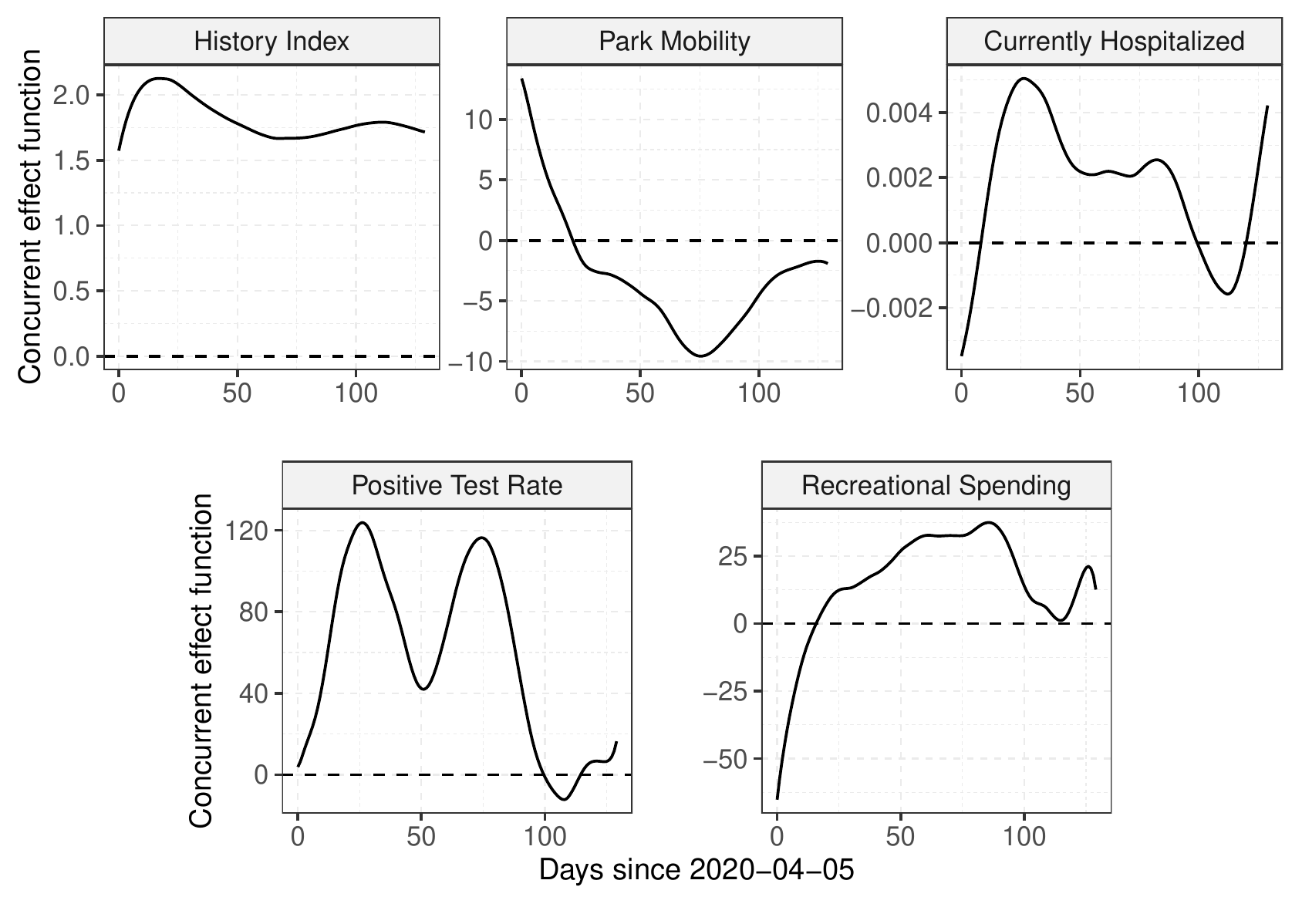}
	\caption{Concurrent effect functions $\hat{\beta}_j$ for predictor processes $U_j$ in the RDED \eqref{eq:fcreg}. 	%The black curves correspond to the regression slopes, and the light gray ribbons correspond to 95\% bootstrap pointwise confidence intervals. Curves corresponding to pointwise significant regression effects at the 5\% level (ribbons do not cover 0) are colored in red.
	}
	\label{fig:beta}
\end{figure}

\subsection{Predictor process effects on growth rate}

The results are visualized in Figure~\ref{fig:beta}. States which are suffering from faster spread of the virus are characterized by higher growth rates.  The historically weighted integrated case count per million has a positive effect on the growth rate. This  reflects dynamic explosive behavior \citep{mull:10}, so 
that states  will tend to have caseloads that move even further away  from the mean caseload per million across states as time moves forward, which could be either moving further above the average across states or further below it.  Thus the positive effect of the history index  reflects increasing  variability in the caseloads across states. %: as time goes on the growth rates expand further and further apart from each other.  
Potential reasons for this are the distinct policies that states implement in terms of mobility, business and social restrictions, as well as cultural acceptance of such restrictions. Spatial infectivity dynamics with waves of infections moving spatially across the U.S. also might play a major role. 

In terms of the effects of covariate processes, our results indicate that higher park mobility during late spring and summer (May onwards) is associated with a decrease in growth rates. For each percent increase in park mobility, the growth rate decreased by as much as 10 cases per million per day, with the effect reaching its maximum in mid-summer.  This indicates negligible risk and potential benefits from outdoors recreation activities. 

The other covariate processes that were selected by LASSO as predictors in the RDED model \eqref{eq:fcreg} are patient counts at hospitals, the testing positivity rate, and credit card spending on recreational activities, including dining at restaurants. All of these were found to  have positive effects on the growth rate throughout nearly the entire time domain. These effects were included in model \eqref{eq:fcreg} with covariate process-specific lags, as per Table~\ref{table:lags}. The number currently hospitalized was found to be most predictive for the case  growth rate when including a  two-week lag, whereas for  test positivity rate a one-week lag was found to be most predictive. This may suggest that individuals who are infected by individuals with less severe symptoms (i.e. a non-hospitalized COVID-positive person) tend to learn of their disease one week after their infector did. 

The two-week difference for hospitalized individuals is likely a reflection of the time needed to develop serious symptoms which warrant admission for medical care. The effect of recreational spending on viral spread was much more immediate; the optimal lag in this case is only 3 days. This suggests that states with higher recreational economies during the pandemic see an almost immediate uptick in growth rate, with cases increasing by as much as 25 infections per million per day for each percent of additional recreational spending. 

%\red{this needs to be expanded and specifically the lags need to be explained}. 

%A natural explanation for the lack of significance among these variables is that they are highly collinear and thus may suffer from variance inflation when included in the model together. That being said, a lack of significant association does not necessarily imply a lack of predictive power for these covariates, as evidenced by their inclusion by LASSO and the high quality of model-fitted growth rates, discussed in the next Section.

\subsection{Predictive performance of the random differential equation with delay model}

Figures~\ref{fig:fit1}~and~\ref{fig:fit2} display the fitted growth rates provided by the RDED model  \eqref{eq:fcreg},  with covariate specific lags, and  also the results of fitting a differential equation model where the covariates do not have individual lags, i.e. a concurrent RDE model given by  \eqref{eq:fcreg} in which $\tau_j=0, j=0,1,2,3,4$. Fitted curves are the result of leave-one-out predictions for each state, where the data from the left-out state were not used for fitting the model and then the state's data are plugged into the fitted model to obtain the predicted case growth rate.  These predictions are compared against the actual observed growth rates as obtained by local polynomial regression (see Section \ref{sec:derivest} and Section A.3 of the Appendix).
 %We emphasize that the leave-out prediction does not use any case data for the state where growth rate is predicted, only the covariate data. 
\begin{figure}[H] \vspace{-1in}
	\centering
	\includegraphics[width = 1.03\linewidth]{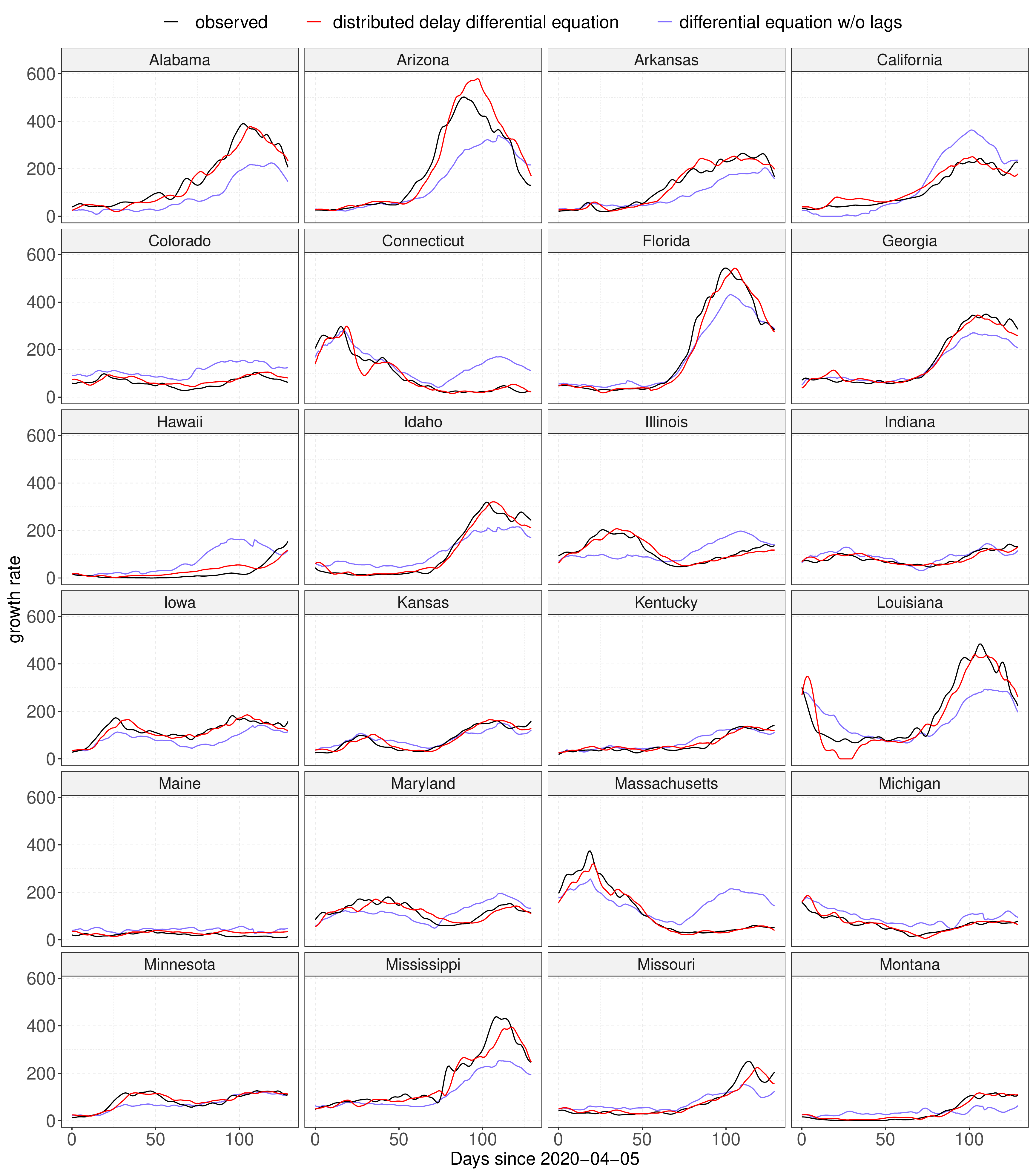}
	\caption{Predicted  growth rates by the proposed  RDED  \eqref{eq:fcreg} (red) and by a RDE without lags (blue) for US States, Alabama - Montana. Observed rates are the black curves.}\label{fig:fit1}
\end{figure}
\begin{figure} \vspace{-1.75in}	\centering
	\includegraphics[width = 1.03\linewidth]{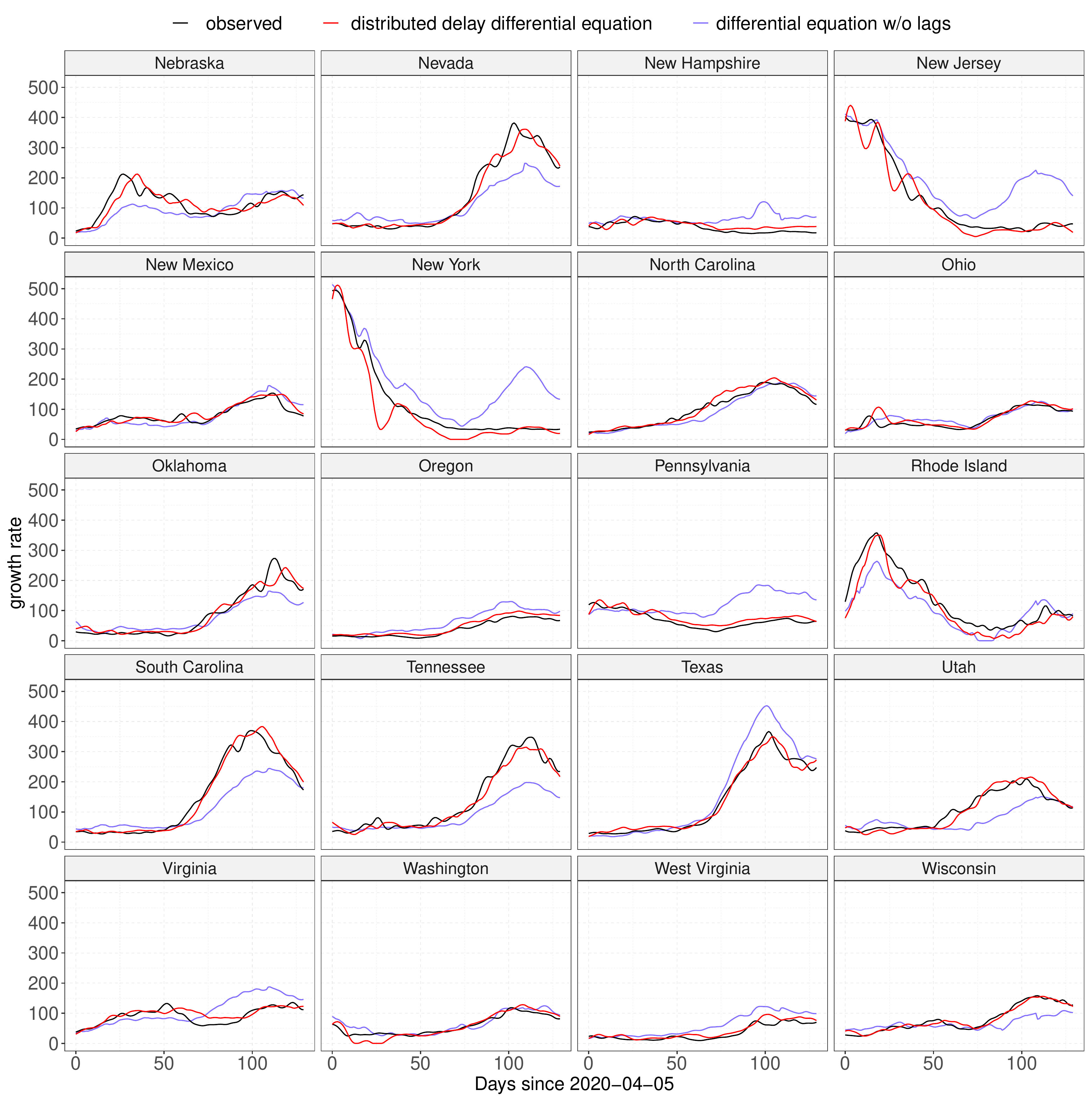}
	\caption{Predicted  growth rates by the proposed  RDED  \eqref{eq:fcreg} (red) and by a RDE without lags (blue) for US States, Nebraska - Wisconsin. Observed rates are the black curves.}	
\label{fig:fit2}
\end{figure}
Visual inspection (as well as measures of goodness-of-fit, such as  integrated mean squared prediction error) indicates that the incorporation of delays with predictor specific lags in the model substantially  improves the fits. This provides compelling evidence for modeling COVID-19 growth rates with RDED that include covariate-specific lags. A vanilla RDE without predictor delays does not incorporate historical information beyond the previous day and struggles to explain deviations in phase (as in, e.g. Illinois). The problem of modeling phase variation is instead foisted upon the coefficient functions, which are ill-equipped to reflect it, and thus at times the straight RDE  model without predictor delays falters in terms of amplitude modeling as well (e.g. for Colorado).
In contrast, the proposed RDED incorporates past information through the included delays and thus reflects  phase variability in a natural way, especially through the inclusion of distributed delays via the history index. % and shifts in the time domain of the other predictors. 
Including a two-week history of case counts through the history index function $\gamma$ in model \eqref{eq:fcreg} substantially enhances the flexibility of the model.  The estimates of coefficient functions are no longer uncontrollably altered  by the effects of unaccounted lags, and as a result the fitted model  performs much better.

\subsection{Model-benchmarked evaluation of performance}

A comparison of observed  vs. RDED model-predicted growth rates allows one to assess the performance of states in terms of controlling spread. A state whose observed growth rate falls below the predicted curve corresponds to a negative residual and suggests better-than-expected viral control, and vice versa for observed rates which fall above the predicted curve. Figures \ref{fig:res1} and \ref{fig:res2} display the residual curves for the fitted  RDED model  \eqref{eq:fcreg}, where positive residuals correspond to the number of excess new cases per day relative to the predictions obtained from fitting  model  \eqref{eq:fcreg}, as described above, including the  two-week history index and other covariate processes  with their respective lags. These residuals reflect the unexplained residual stochastic process $Z$ that is an integral and principally unpredictable part of the RDED model 
 \eqref{eq:fcreg} on which our predictions are based.
 
{Major outbreaks, which can be interpreted as an unforeseen amount of excess cases, are revealed through positive spikes in the residual plots. For example, New York's residuals leap up to as high as 170 excess cases per million per day during a spike in Spring.  Arizona's summer outbreak is captured in the residuals as well, subsequently course-correcting with a streak of fewer-than-expected new cases around day 100. These spikes also suggest that in the two weeks prior to the uptick, there were substantially more cases in the state population than the numbers reported in the JHU dataset.}

{The consistency of a state's handling of the virus can be seen through the volatility of the residual curve. States like Colorado and Oregon exhibit only minor deviations from zero, which suggests their new cases are quite  predictable given a two-week history of the state case counts and covariate processes.  Some states, particularly in the Northeast, including  New York, New Jersey, and Rhode Island,  only exhibit high volatility during the early days of the pandemic before achieving a period of stability. On the other hand, states such as  Alabama and Louisiana exhibit less-predictable case counts throughout. 

%Minnesota and Michigan have negative residuals for a long span of time in the early days which contributes to the overall negative average performance but tends to perform better in controlling the virus later on. Hawaii on the other hand has an overall zero average residual, although the trends between early and later days is reversed in comparison with Michigan and Minnesota. Connecticut, Louisiana, Massachusetts, and Nevada are among the states having overall average behaviour. Wisconsin is the worst performing states according to the model, whereas Pennsylvania has performed the best when considering their predictors. 

\subsection{Computation times and code availability}
The codes for reproducing the above analysis are available at \url{https://osf.io/w48zd/?view_only=7cd3e2c686a74fd086e4eae8534214c0}. The run times of the different steps in the analysis using R version 3.6.3 (2020-02-29) running under CentOS Linux 7 (Core)  on x86\_64-pc-linux-gnu (64-bit) platform are summarized in Table \ref{table:runtime}. 

\begin{figure}[htbp]
	\centering
	\includegraphics[width = 1.05\linewidth]{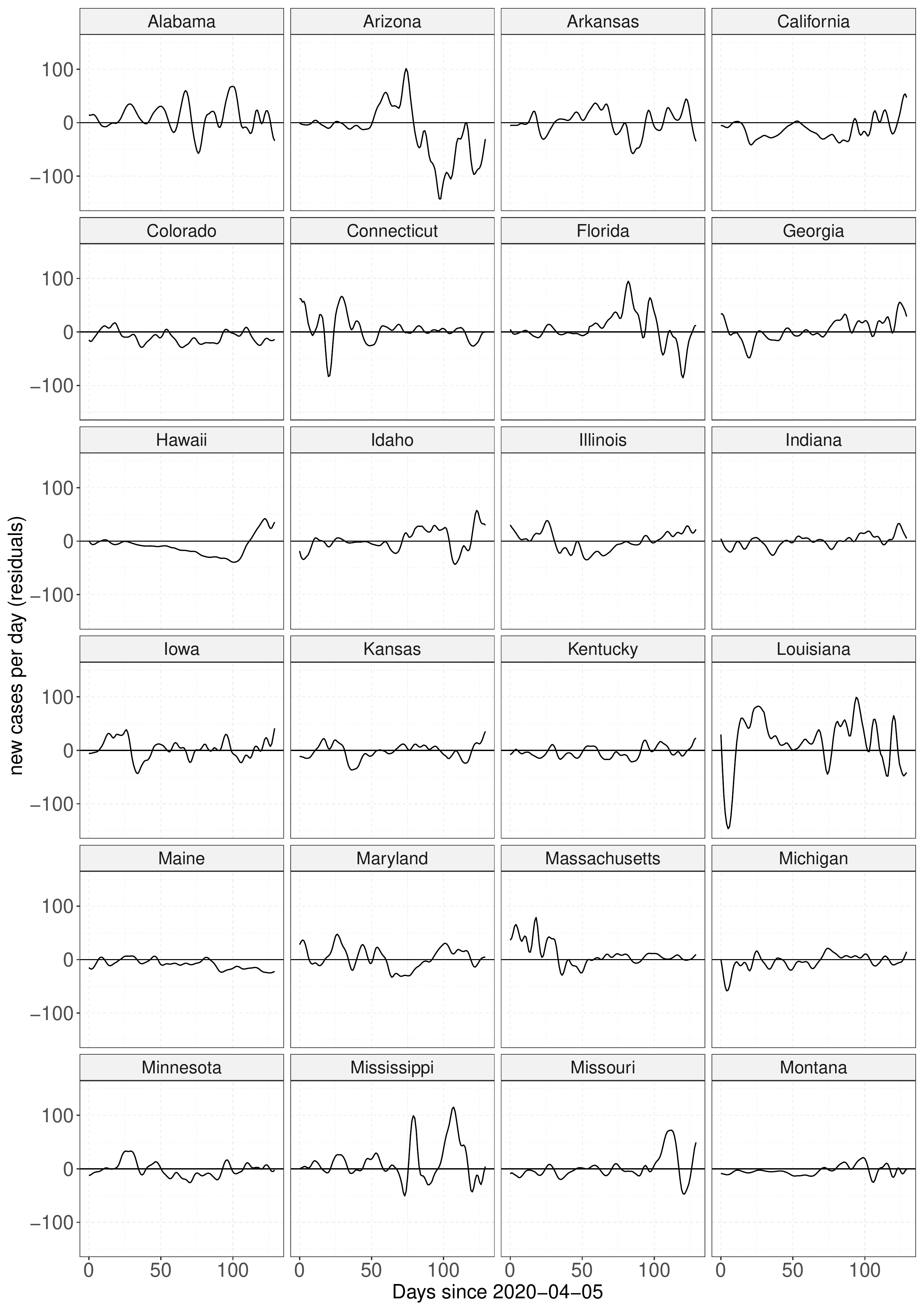}
	\caption{Residuals for the predictions from the RDED  \eqref{eq:fcreg} for US States, Alabama - Montana.}\label{fig:res1}
\end{figure}

\begin{figure}[htbp]
	\centering
	\includegraphics[width = 1.05\linewidth]{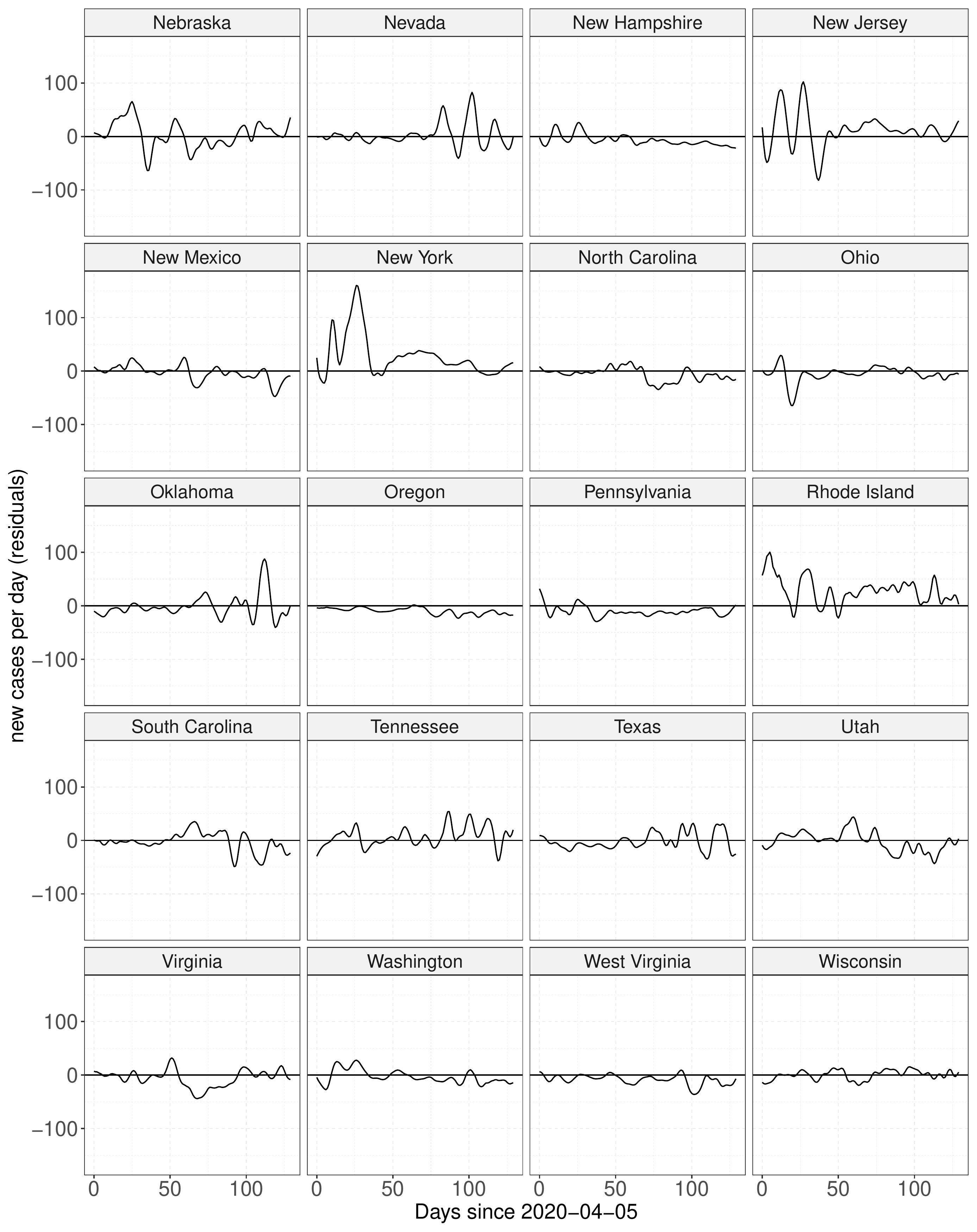}
	\caption{Residuals for US States, Nebraska -  Wisconsin.}\label{fig:res2}
\end{figure}

\begin{table}
\begin{tabular}{|l|l|l|l|}
\hline
Step&Description&Number of Cores&Time in Seconds \\ \hline
1&preprocessing&single core                     & 9.475           \\ \hline
2&initial lag selection&8 cores in parallel             & 422.505         \\ \hline
3&history index function estimation&single core & 3.414           \\ \hline
4&lasso&8 cores in parallel                             & 31.316          \\ \hline
5&backfitting&single core                       & 2565.715        \\ \hline
6&smoothing&single core                         & 0.012           \\ \hline
&Total&                         & 3032.437           \\ \hline
\end{tabular}
\caption{Run time for the different steps in the data analysis.}
\label{table:runtime}
\end{table}

\section{Discussion and concluding remarks}

In this paper we propose dynamics learning exemplified by learning a RDED from a sample of observed trajectories, which is motivated by modeling COVID-19 daily new cases from the previous behavior of the case trajectory itself through a distributed delay or history index component and additional covariate process components.   RDEDs form the driving mechanism of  real life processes that include a  feedback loop  into the future. Depending on the application and the exact nature of the feedback, one might incorporate discrete delays or distributed delays. While discrete delays borrow information from some isolated times in the past, a distributed delay takes into account the entire continuum in the recent past and turns out to be closely related to a functional history index model that has been considered in the statistics literature. 

In the current framework, we have not included  inference for the estimated model coefficient functions and the lags. It is left for future research to investigate the convergence of these estimated parameters to their population targets and to obtain guarantees on the  convergence rates. Some words of caution are in order. Any method that includes a time delay component suffers from the limitation that a larger set of possible lags means one has to sacrifice a part of the data that are available for the analysis, due to the initialization.   We have assumed that the trajectories corresponding to the different states are i.i.d, realizations, which may not hold if there is significant spatial correlation based on geographic and demographic similarities across the states. It would be interesting to study adaptations of the proposed methodology when the samples are correlated instead of being independent. From an application point of view, a different and perhaps more refined perspective could potentially be gained by applying the proposed method at a finer level to county level data.

Establishing existence and uniqueness of the solutions of RDEDs has found increasing interest in recent years \citep{cala:19,cort:20}. We contribute here to the forward problem of obtaining solutions of an RDED by targeting the inverse problem analytically and combine this with functional data analysis approaches to estimate the parameter functions from an observed  sample of multivariate stochastic processes for both discrete and distributed delay setups. Our approach builds on and contributes to empirical dynamics and the literature on inferring dynamics and differential equations  from functional data, and extends it  to models that  include time delays.  

A key feature of the proposed approach  is the statistical estimation of the model parameters in the population RDED, which has inbuilt model selection steps that aid in  estimating the optimal lags and also to select the important predictors through the lasso pruning steps, with majority voting across all time points. This reduces the number of predictors in the model and thus its complexity before the optimization step of the lags that are associated with the predictors. 
To handle this complex lag optimization step,  instead of a brute force search method, we adopted a backfitting algorithm, which significantly simplifies the optimization step and also leads to faster convergence. We provide a complete toolkit that starts with a sample of multivariate trajectories and extracts from the observed data an estimated RDED that best explains the dynamics inherent in the data, along with automatic variable selection and delay estimation.

The application of our method to the modeling of COVID 19 growth rates in the United States demonstrates excellent performance in terms of predictions of the trajectories. The predictions from the random differential equation model without predictor lags grossly overestimates the COVID-19 growth rates for significant periods of time in New York, New Jersey, Pennsylvania, Massachusetts, Connecticut, Virginia, West Virginia, New Hampshire, California, Texas, Colorado, Oregon, Washington, Maine and Hawaii but underestimates the growth rates for Alabama, Arizona, Arkansas, Florida, Louisiana, Nebraska, Nevada, South Carolina, Tenessee and Wisconsin; see Figures \ref{fig:fit1} and \ref{fig:fit2}. This highlights the importance  of incorporating time lags in differential equation modeling of functional data, when one has a sample of realized stochastic processes, specially when the underlying processes tend to have an aftereffect. %, a behaviour commonly observed in epidemic modeling. 

%\section{Conclusion}

%\noindent{\bf APPENDIX} 
\appendix

\renewcommand{\thesection}{A.\arabic{section}}

\section*{Appendix}

%\addtocontents{toc}{\protect\setcounter{tocdepth}{-1}}

\section{Proof of the existence of a solution for the model in \texorpdfstring{\eqref{CONC:DEF}}{}}
\begin{proof}[Proof of continuity of $f$]
The unique solution of the first order  DDE in \eqref{CONC:DEF} can be shown to exist as a direct extension of Theorem 2.2 of \cite{cala:19}.

We first show that $f(y,t)$ is $p^{\text{th}}$-moment continuous,
%, that is, for any sequence $t_m \rightarrow t \in I$ with $t_m \in I$, and a sequence of random variables $y_m \rightarrow y$ in the $p^{\text{th}}$-moment,  $$||f(y_m,t_m) - f(y,t)||_p := (E|f(y_m,t_m) - f(y,t)|^p)^{1/p} \rightarrow 0.$$ 
using assumptions \ref{ass:A1} - \ref{ass:A3}  in Section \ref{model:conc}. Observe that
\begin{align}
\label{conc:proof1}
&(E|f(y_m,t_m) - f(y,t)|^p)^{1/p} \nonumber\\ 
& = ( E|\alpha(t_m) -  \alpha(t) \ + \beta_0(t_m)y_m - \beta_0(t)y \ + Z(t_m)-Z(t) \nonumber\\ & \ + \sum_{j=1}^J\beta_{j}(t_m)U_j(t_m-\tau_{j}) -  \sum_{j=1}^J\beta_{j}(t)U_j(t-\tau_{j})|^p)^{1/p}\nonumber \\
& \leq |\alpha(t_m) -  \alpha(t)|\ + ||\beta_0(t_m)y_m - \beta_0(t)y||_p + ||Z(t_m)-Z(t)||_p \nonumber\\ & \ + \sum_{j=1}^J ||\beta_{j}(t_m)U_j(t_m-\tau_{j}) -  \beta_{j}(t)U_j(t-\tau_{j})||_p.
\end{align}
The first and third terms in \eqref{conc:proof1} converge to  $0$ using assumptions \ref{ass:A1} and \ref{ass:A2} respectively. For the second and fourth terms of \eqref{conc:proof1}, 
\begin{align}
\label{conc:proof2}
    &||\beta_0(t_m)y_m - \beta_0(t)y||_p \nonumber\\
    &\leq ||\beta_0(t_m)\ y_m - \beta_0(t_m)\ y||_p \ + ||\beta_0(t_m) \ y - \beta_0(t)\ y||_p \nonumber \\
    &\leq |\beta_0(t_m)| \ ||y_m -y ||_p \ + |\beta_0(t_m) - \beta_0(t)|\ ||y||_p.
   % &\leq (|\beta_0(t_m) -\beta_0(t)| \ +|\beta_0(t)|) \ ||y_m -y ||_p \ + |\beta_0(t_m) - \beta_0(t)|\ ||y||_p.
 \end{align}
By assumption \ref{ass:A1}, $\beta_0(t_m)$ is continuous on the compact domain $\mathcal{T}$, and hence uniformly bounded. Also, by assumption $y_m \rightarrow y$, in the $p^{\text{th}}$-moment. Since $\Lp$, the space of random variables with finite $p^{\text{th}}$ moments, is a Banach space, it follows that $||y||_p := (E|y|^p)^{1/p} < \infty.$ Thus \eqref{conc:proof2} converges to $0$. Next
\begin{align}
\label{conc:proof3}
&\sum_{j=1}^J ||\beta_{j}(t_m)U_j(t_m-\tau_{j}) -  \beta_{j}(t)U_j(t-\tau_{j})||_p \nonumber\\
&\leq \sum_{j=1}^J ||\beta_{j}(t_m) \ (U_j(t_m-\tau_{j}) - U_j(t-\tau_{j})) \nonumber \\ & \hspace{1cm} + (\beta_{j}(t_m) - \beta_{j}(t)) \ U_j(t-\tau_{j})||_p \nonumber \\
&\leq \sum_{j=1}^J |\beta_{j}(t_m)| \ ||U_j(t_m-\tau_{j}) - U_j(t-\tau_{j})||_p \nonumber \\ & \hspace{1cm} + |\beta_{j}(t_m) - \beta_{j}(t)| \ ||U_j(t-\tau_{j})||_p,
\end{align}
and \eqref{conc:proof3} converges to $0$ using the assumptions on the continuity of the coefficient functions $\beta_{j}(\cdot)$ on the compact domain $\mathcal{T}$. This implies the uniform boundedness of $\beta_{j}$
and the continuity of the predictor processes $U_{j}(\cdot) \in \Lp$  in the $p^{\text{th}}$- moment for $j = 1,\dots, J.$ 
Hence the proof of continuity of $f(y,t)$ follows immediately by combining \eqref{conc:proof1}, \eqref{conc:proof2}, and \eqref{conc:proof3}.
\end{proof}

\begin{proof}[Proof of Proposition \ref{thm1}]
Let $k(t) = \beta_0(t) ,\ t \in \mathcal{T}$ Since, by assumption \ref{ass:A1},  $\beta_0(\cdot)$ is continuous on a compact domain, we have $\int_{t_0}^T |k(t)| <\infty.$
For $y,v \in \Lp$ and $t \in \mathcal{T}$, we obtain
\begin{align}
    ||f(y,t) - f(v,t)||_p & = ||\beta_0(t) \ (y-v)||_p \leq |\beta_0(t)| ||y-v||_p.
\end{align}
Since $f$ is $p^{\text{th}}$-moment continuous, the proof follows using the same line of arguments as in \cite{cala:19}.
\end{proof}

\section{Proof of the existence of a solution for the model in  \texorpdfstring{\eqref{histIndex1}}{}}
For ease of presentation, we will begin by showing the existence of the unique solution of  the first order RDED \eqref{histIndex1} with one predictor process $U(\cdot)$, rewritten as in \eqref{histIndex1:Calatayud2}. 
The arguments can then be easily extended to the case of multiple predictor processes, and we omit the details.
%For a more general history index model in \eqref{histIndex2}, with $p$ predictor processes $U_j(\cdot), \ p \geq 1$, the proofs can be readily adapted by using the same line of arguments.

\begin{proof}[Proof of Proposition \ref{prop:cont:hist}]
Recall that $t_m \in [t_0,b]$ is a real sequence converging to $t\in [t_0,b]$ and $Y$ is an $\Lp$-continuous process, so that
 \begin{align}
\label{hist1:proof1}
&||\tilde{f}(Y,t_m) - \tilde{f}(Y,t)||_p \nonumber\\
= & ||\alpha(t_m) - \alpha(t)   +   \int_{t_m-\tau_0}^{t_m}\gamma(t_m-s,t_m) Y(s)ds - \ \int_{t-\tau_0}^{t}\gamma(t-s,t) Y(s)ds \nonumber \\ + &  \ \int_{t_m-\tau_1}^{t_m} \gamma_1(t_m-s,t_m) U(s)ds  - \int_{t-\tau_1}^{t} \gamma_1(t-s) U(s)ds \nonumber +  Z(t_m) - Z(t)||_p\nonumber \\
\vspace*{3cm}
\leq & |\alpha(t_m) -\alpha(t)| + ||Z(t_m)- Z(t)||_p \nonumber \\
& \ + ||\int_{t_m-\tau_0}^{t_m}\gamma(t_m-s,t_m) Y(s)ds -  \int_{t-\tau_0}^{t}\gamma(t-s,t) Y(s)ds ||_p \nonumber \\ 
& \ + ||\int_{t_m-\tau_1}^{t_m} \gamma_1(t_m-s,t_m) U(s)ds -  \int_{t-\tau_1}^{t} \gamma_1(t-s,t) U(s)ds ||_p.
 \end{align}

Using assumptions \ref{ass:A1} and \ref{ass:A2}, respectively, the first and second term of \eqref{hist1:proof1} converge to $0$. As before, we consider the third and fourth terms separately, 
\begin{align}
    \label{hist:proof2}
&||\int_{t_m-\tau_0}^{t_m}\gamma(t_m-s,t_m) Y(s)ds -  \int_{t-\tau_0}^{t}\gamma(t-s,t) Y(s)ds ||_p \nonumber\\
\leq & ||\int_{t_m-\tau_0}^{t_m}\gamma(t_m-s,t_m) Y(s)ds -  \int_{t_m-\tau_0}^{t_m}\gamma(t-s,t) Y(s)ds ||_p \nonumber \\
& + || \int_{t_m-\tau_0}^{t_m}\gamma(t-s,t) Y(s)ds -  \int_{t-\tau_0}^{t}\gamma(t-s,t) Y(s)ds ||_p \nonumber \\
\leq & || \int_{t_m-\tau_0}^{t_m}(\gamma(t_m-s,t_m) - \gamma(t-s,t)) Y(s)ds ||_p \nonumber \\ 
& + \ || \int_{t_m-\tau_0}^{t_m}\gamma(t-s,t) Y(s)ds -\int_{t-\tau_0}^{t} \gamma(t-s,t) Y(s)ds ||_p.
\end{align}
For the first term in \eqref{hist:proof2},
\begin{align}
    \label{hist1:proof3}
&|| \int_{t_m-\tau_0}^{t_m}(\gamma(t_m-s,t_m) - \gamma(t-s,t)) Y(s) \ ds ||_p \nonumber \\
& \leq \int_{t_m-\tau_0}^{t_m} || (\gamma(t_m-s,t_m) - \gamma(t-s,t)) Y(s)||_p \ ds\nonumber \\
& \leq \int_{t_0-\tau_0}^{b} || (\gamma(t_m-s,t_m) - \gamma(t-s,t)) Y(s) ||_p \ ds\nonumber \\
& =\int_{t_0-\tau_0}^{b} |(\gamma(t_m-s,t_m) - \gamma(t-s,t))| \ || Y(s)||_p \ ds,
\end{align}
where the first inequality follows from the fact that the integrand is $p^{\text{th}}$- moment continuous, using assumption \ref{ass:B1} and the assumption that $Y(\cdot)\in \Lp$ (see \cite{soon:73} page 102(3)). The second inequality follows using the properties of $\Lp$ Riemann integration. Now, from assumption \ref{ass:B1}, we have that for any $\epsilon>0$ there exists $M_0(\epsilon)$ such that $|(\gamma(t_m-s,t_m) - \gamma(t-s,t))| < \varepsilon$ whenever $m \geq M_0(\epsilon)$. Thus, from \eqref{hist1:proof3} it follows that 
\begin{align}
    \label{hist1:proof4}
\int_{t_0-\tau_0}^{b} || (\gamma(t_m-s,t_m) - \gamma(t-s,t)) Y(s)||_p \ ds \leq \int_{t_0-\tau_0}^{b} \epsilon ||Y(s)||_p \ ds.
\end{align}

Since $||Y(s)||_p < \infty$ and $[t_0-\tau_0,b]$ is compact, to show that $\int_{t_0-\tau_0}^{b}||Y(s)||_p ds<\infty$ it is sufficient to prove that the application  $s\rightarrow ||Y(s)||_p$ is continuous on $[t_0-\tau_0,b]$. For this, let $s_m \rightarrow s$ as $m\rightarrow \infty.$ By assumption \ref{ass:B2} we have
\begin{align*}
    |||Y(s_m)||_p - ||Y(s)||_p| \leq ||Y(s_m) - Y(s)||_p \rightarrow 0 \text{ as } m \rightarrow \infty.
\end{align*}
Hence
\begin{align}
   \label{hist1:proof5}
  || \int_{t_m-\tau_0}^{t_m}(\gamma(t_m-s,t_m) - \gamma(t-s,t)) Y(s)ds ||_p \rightarrow 0, \text{ as } m \rightarrow 0.
\end{align}
%Recall the second term of \eqref{hist:proof2}
%$$|| \int_{t_m-\tau_0}^{t_m}\gamma(t-s,t) Y(s)ds -\int_{t-\tau_0}^{t} \gamma(t-s,t) Y(s)ds ||_p.$$
For the function $\Psi(r) := \int_{r-\tau_0}^{r}\gamma(t-s,t) Y(s)ds$, $r \in [t_0,b]$,
\begin{align*}
     \Psi(r) = \int_{r-\tau_0}^{r}\gamma(t-s,t) Y(s)ds & \\= \int_{t_0-\tau_0}^{r}\gamma(t-s,t) Y(s)ds & - \int_{t_0-\tau_0}^{r-\tau_0}\gamma(t-s,t) Y(s)ds\\ =
    \int_{t_0-\tau_0}^{r}\gamma(t-s,t) Y(s)ds & - \int_{t_0}^{r}\gamma(t-u+\tau_0,t) Y(u-\tau_0)du.
\end{align*}
From \cite{soon:73} page 103(5), each term in the RHS of the above equation is $p^{\text{th}}$-moment continuous. Hence, by definition, for $t_m \rightarrow t,\ ||\Psi(t_m)- \Psi(t)||_p = || \int_{t_m-\tau_0}^{t_m}\gamma(t-s,t) Y(s)ds -\int_{t-\tau_0}^{t} \gamma(t-s,t) Y(s)ds ||_p \rightarrow 0,$ and the second term in \eqref{hist:proof2} also converges to zero as $m\rightarrow \infty.$
As for the convergence of the remaining fourth term, involving the predictor process $U(\cdot)$ in \eqref{hist1:proof1}, we can follow a similar line of argument to obtain 
\begin{align}
    \label{hist1:proof6}
& ||\int_{t_m-\tau_1}^{t_m} \gamma_1(t_m-s,t_m) U(s)ds -  \int_{t-\tau_1}^{t} \gamma_1(t-s,t) U(s)ds ||_p \nonumber \\
\leq & \ ||\int_{t_m-\tau_1}^{t_m} (\gamma_1(t_m-s,t_m)-\gamma_1(t-s,t)) U(s)ds||_p \nonumber \\
& + \ || \int_{t_m-\tau_1}^{t_m} \gamma_1(t-s,t) U(s)ds - \int_{t-\tau_1}^{t} \gamma_1(t-s,t) U(s)ds||_p\nonumber\\
& \rightarrow 0, \text{ using assumptions \ref{ass:A3} and \ref{ass:B1}}.
\end{align}
Thus, $(E|\tilde{f}(Y,t_m) - \tilde{f}(Y,t)|^p)^{1/p} \rightarrow 0$, for a sequence $t_m \rightarrow t$ as $m \rightarrow \infty,$ implying the $p^{\text{th}}$-moment continuity of $\tilde{f}(Y,t)$ for all $t \in I$. The $p^{\text{th}}$-moment continuity of the process $\tilde{f}(Y,\cdot)$ follows by definition.
\end{proof}

As a consequence of the above result, we observe that $\tilde{f}(X,t)$ in \eqref{histIndex1:Calatayud2} is continuous in $t \in \mathcal{T} = [t_0,T]$ in the $\Lp$ sense, by simply choosing $b= T$.

\begin{proof}[Proof of Proposition \ref{prop:Lpsol:hist}]
Suppose $X$ is an $\Lp$ solution of \eqref{histIndex1:Calatayud2}. Then by definition, \ref{cond:a} and \ref{cond:c} hold. Since $X$ is $p^{\text{th}}$-moment continuous (see assumption \ref{ass:B2}) and the process $\tilde{f}(X,\cdot)$ is $\Lp$-continuous on $\mathcal{T} = [t_0,T]$, it follows from \cite{soon:73} page 101(1) that $X'$ is $p^{\text{th}}$-moment continuous and $p^{\text{th}}$-moment Riemann integrable on $\mathcal{T}$. Thus from the Fundamental Theorem of $\Lp$ Calculus (\citep{soon:73} page 104(6))
\begin{align*}
    X(t) &= g(t_0) + (\Lp) \int_{t_0}^{t} X'(s) ds
    = g(t_0) + \int_{t_0}^{t} \tilde{f}(X,s)ds,
\end{align*}
and condition \ref{cond:b} holds.
On the other hand, if conditions \ref{cond:a}-\ref{cond:c} are true, $X$ is indeed an $\Lp$ solution of \eqref{histIndex1:Calatayud2},  since the process $\tilde{f}(X,\cdot)$ was shown to be continuous. Thus $\tilde{f}(X,\cdot)$ is $p^{\text{th}}$-moment Riemann integrable on $\mathcal{T}$ (\citep{soon:73} page 101 (1)), and $\int_{t_0}^{t} \tilde{f}(X,s)ds$ is well defined in the $\Lp$ sense. From the result on page 103 (5) of \cite{soon:73}, it follows that $X$ is $p^{\text{th}}$-moment  differentiable on $\mathcal{T}$, with $X'(t) = \tilde{f}(X,t).$ Hence the proposition follows.
\end{proof}

\begin{proof}[Proof of Theorem \ref{thm2}]
From the history index model in \eqref{histIndex1:Calatayud2}
\begin{equation}
\begin{aligned} X'(t) = \tilde{f}(X, t)=&\alpha(t)+\int_{t-\tau_0}^t\gamma(t-s,t)X(s)ds\\&+ \int_{t-\tau_1}^{t} \gamma_1(t-s,t) U(s)ds +Z(t).
\end{aligned}
\nonumber
\end{equation}
For $k(t)=\int_{t-\tau_0}^t|\gamma(t-s,t)|ds$, since $\gamma$ is continuous on a compact domain, hence uniformly bounded, we have $k\in L^1(\mathcal{T})$.
Following the arguments outlined in the proof of Theorem 2.2 in \cite{cala:19}, this implies  $\lim_{t\rightarrow t_0^+}\int_{t_0}^tk(s)ds=0$. Thus, we can choose $\alpha>t_0$ such that, for all $t\in \mathcal{T}_{\alpha}=[t_0, \alpha]\subset \mathcal{T}, \int_{t_0}^tk(s)ds\leq 1/2.$ Consider the vector space
$$\mathcal{A}=\{X: [t_0-\tau_0, \alpha]\rightarrow \mathcal{L}^p\text{ continuous}, X(t)=g(t)\text{ on }[t_0-\tau_0, t_0]\},$$
with  norm
$\Vert X\Vert_{\mathcal{A}}=\sup_{t\in[t_0-\tau_0, \alpha]}\Vert X(t)\Vert_p.$
We note that $\Vert X\Vert_{\mathcal{A}}$ is well-defined because, by the $p^{\text{th}}$-moment continuity of $X$, the real map $t \rightarrow\Vert X(t)\Vert_p,\ t \in[t_0-\tau_0, \alpha]$ is continuous. Therefore $\underset{t\in[t_0-\tau_0, \alpha]}{\sup}\ \Vert X(t)\Vert_p<\infty$.

As shown in the proof of Theorem 2.2 of \cite{cala:19}, $\mathcal{A}$ is a Banach space. Consider the map $\Lambda: \mathcal{A}\rightarrow\mathcal{A}$ such that

$$\Lambda(X)(t)=\begin{cases}g(t_0)+(\Lp)\int_{t_0}^t\tilde{f}(X, s)ds, & t\in \mathcal{T}_{\alpha}=[t_0, \alpha],\\g(t), & t\in[t_0-\tau_0, t_0].\end{cases}$$

\noindent If $X\in\mathcal{A}$, it follows by taking $b= \alpha$ in Proposition \ref{prop:cont:hist} that the process $\tilde{f}(X, t)$ is $p^{\text{th}}$-moment continuous over $t \in \mathcal{T}_\alpha$. From arguments on  p. 103 (5) of \cite{soon:73}, the $p^{\text{th}}$-moment Riemann integral $(\Lp)\int_{t_0}^t\tilde{f}(X, s)ds$ is continuous. Also, the initial condition $g(\cdot)$ is continuous on $[t_0-\tau_0,t_0]$. Thus $\Lambda X: [t_0-\tau_0, \alpha]\rightarrow \Lp$ is continuous on $[t_0-\tau_0, \alpha]$ satisfying $\Lambda X(t)=g(t)$ on $[t_0-\tau_0, t_0]$. By definition $\Lambda X \in \mathcal{A}$, so $\Lambda$ is well defined.

Since from Proposition \ref{prop:Lpsol:hist}, $X: \mathcal{T}_{\alpha}\rightarrow\Lp$ is a solution of equation \eqref{histIndex1:Calatayud2} if and only if $X\in\mathcal{A}$ and $\Lambda X=X$, by the Banach fixed-point theorem, it suffices to check that $\Lambda$ is a contraction. Let $X, Y\in\mathcal{A}$ and $t\in \mathcal{T}_\alpha$  and observe that if $t\in[t_0-\tau_0, t_0]$ then  $\Lambda X(t)-\Lambda Y(t)=g(t)-g(t)=0$). From arguments on p. 102 (3) of \cite{soon:73}
\begin{equation}
    \begin{aligned}
    \Vert\Lambda X(t)-\Lambda Y(t)\Vert_p&=\left\Vert(\Lp)\int_{t_0}^t(\tilde{f}(X, s)-\tilde{f}(Y, s))ds\right\Vert_p\\&\leq\int_{t_0}^t\Vert \tilde{f}(X, s)-\tilde{f}(Y, s)\Vert_pds\\&=\int_{t_0}^t\left\Vert \int_{s-\tau_0}^s\gamma(s-u,s)(X(u)-Y(u))du\right\Vert_pds\\&\leq \int_{t_0}^t\int_{s-\tau_0}^s\vert\gamma(s-u,s)\vert\Vert X(u)-Y(u)\Vert_pduds\\&\leq\int_{t_0}^t\int_{s-\tau_0}^s\vert\gamma(s-u,s)\vert duds\Vert X-Y\Vert_{\mathcal{A}}\\&=\int_{t_0}^tk(s)ds\Vert X-Y\Vert_{\mathcal{A}}\\&\leq\frac{1}{2}\Vert X-Y\Vert_{\mathcal{A}}.
    \end{aligned}
    \nonumber
\end{equation}
Taking the supremum on $t\in [t_0-\tau_0, \alpha]$, $\Vert\Lambda X-\Lambda Y\Vert_{\mathcal{A}}\leq \frac{1}{2}\Vert X-Y\Vert_{\mathcal{A}}$. Therefore $\Lambda$ is a contraction.
The remainder of the proof follows similar arguments as in the proof of Theorem 2.2 of \cite{cala:19}.
%The results of Propositions \ref{prop:cont:hist} and \ref{prop:Lpsol:hist} hold even when we consider $J$ many predictors as in the history index model \eqref{histIndex2} by including regularity condition of such processes. The functional predictors $U_j(\cdot), \ j=1,\dots J$, are only a function of $t$ hence, under assumptions \ref{ass:A1}-\ref{ass:A3}, and \ref{ass:B1}-\ref{ass:B2}, complies with all of the above arguments. Thus, Theorem \ref{thm2} holds for the more general history index model given by \eqref{histIndex2}.
\end{proof}
{
\begin{proof}[Proof of Corollary \ref{cor1}]
Similar to \eqref{def:f:hist}, define $\tilde{h}: \Lp \times I_b \rightarrow \mathbb{R}$ such that for an $\Lp$-continuous function $Y(\cdot) : [t_0-\tau_0,b]\rightarrow \mathbb{R}$ and $t\in I_b := [t_0,b]$, $t_0<b,$
\begin{align*}
\tilde{h}(Y,t)&= \alpha(t) + \ \int_{t-\tau_0}^{t}\gamma(t-s,t) Y(s)ds +\sum_{\pidx = 1}^{\nprdt} \slope_{\pidx}(\tm) \prdt_{\pidx}(\tm-\lag_{\pidx})  +Z(t).
\end{align*}
Here $\tilde{h}$ is a random functional with first argument given by  trajectories $\{Y(s) : s \in [t_0-\tau_0,b]\}$ and  second argument  $t \in [t_0,b]$. Note that $\tilde{h}$ constitutes a distributional delay term on $Y$ and a discrete concurrent delay on the predictors $U_j$. Thus
\begin{align*}
&||\tilde{h}(Y,t_m) - \tilde{h}(Y,t)||_p \nonumber\\
\leq & |\alpha(t_m) -\alpha(t)| + ||Z(t_m)- Z(t)||_p \nonumber \\
& \ + ||\int_{t_m-\tau_0}^{t_m}\gamma(t_m-s,t_m) Y(s)ds -  \int_{t-\tau_0}^{t}\gamma(t-s,t) Y(s)ds ||_p \nonumber \\ 
& \ + ||\sum_{j=1}^J \beta_{j}(t_m)U_j(t_m-\tau_{j}) -  \beta_{j}(t)U_j(t-\tau_{j})||_p.
 \end{align*}
By similar arguments as in the proof of the continuity of $f$  in Appendix A and the proof of Proposition \ref{prop:cont:hist} in Appendix B, the continuity of $\tilde{h}$ in the $\Lp$ sense in $t$ follows. The existence and uniqueness follows by the same arguments as  in the proof of  Theorem \ref{thm2}.
\end{proof}}

\section{Derivative estimation}
\label{sec:derivestappendix}

% Flexible nonparametric methods of derivative estimation include local polynomial-based \citep{fan:96}, difference quotient-based \citep{mull:87}, and principal component-based  \citep{liu:09, dai:18} approaches. 

For derivative estimation,  we employ the local polynomial estimator for derivatives, motivated by local polynomial approximation \citep{mull:87:4, fan:96}, which leads to the weighted least squares estimates $\hat{\theta}_l$ that correspond to  solving 
\begin{equation}
\argmin_{\theta_l\in\mathbb{R}}\sum_{k=1}^K\{Y_k-\sum_{l=0}^L\theta_l(t_k-t)^l\}^2K_h(t_k-t),
%\label{locpoly1}
\end{equation}
where $K$ is a kernel function with $K_h(\cdot)=K(\cdot/h)/h$ and $h$ is a tuning parameter.  A well-studied estimator for the $\nu$-th order derivative $x^{(\nu)}(t)$ is given by
\begin{equation}
\hat{x}^{(\nu)}(t)=\nu!\hat{\theta}_{\nu},
%\label{locpoly2}
\end{equation}
for $\nu=0, 1, \dots, L$. The whole curve $\hat{x}^{(\nu)}(\cdot)$ is obtained by running the above local polynomial regression with $t$ varying in an appropriate estimation domain. 

For local polynomial fitting $L-\nu$ preferably is taken to be odd as shown in \citet{ruppert1994multivariate} and \citet{fan:96}.
To obtain an estimate of the  first  derivative, i.e., for $\nu=1$,  the choice $L=\nu+1=2$  leads to the so-called local quadratic regression and the derivative estimate $\hat{x}^\prime(t)$ is given by the local slope $\hat{\theta}_1$.
Other common methods for derivative estimation can be based 
on smoothing splines or B-splines \citep{rice:83, zhou:00}. 
An alternative method is based on difference quotients, which provides a straightforward approach for pointwise estimation of derivatives. Difference quotient-based estimators have been thoroughly studied in the context of human growth curves in the nonparametric regression literature \citep{mull:87, mull:87:4, gass:84}. %To give a brief summary,  difference quotients are first 

\section*{References}

\unappendix

\bibliography{covidDEmodeling}

\end{document}